\documentclass[12pt,a4paper,reqno]{amsart}
\usepackage[margin=2.5cm]{geometry}
\usepackage[T1]{fontenc}
\usepackage{amssymb,mathrsfs,bm,dsfont,stmaryrd}
\usepackage{graphicx}
\usepackage{enumitem}
\setlist[enumerate,1]{label=(\roman*)}
\usepackage{cite}
\usepackage[hidelinks]{hyperref}

\numberwithin{equation}{section}
\allowdisplaybreaks[1]

\usepackage{cleveref}
  \crefname{theorem}{Theorem}{Theorems}
  \crefname{theorem}{Theorem}{Theorems}
  \crefname{theorem*}{Theorem*}{Theorems}
  \crefname{lemma}{Lemma}{Lemmas}
  \crefname{lemma}{Lemma}{Lemmas}
  \crefname{remark}{Remark}{Remarks}
  \crefname{proposition}{Proposition}{Propositions}
\crefname{notation}{Notation}{Notations}
\crefname{claim}{Claim}{Claims}
  \crefname{definition}{Definition}{Definitions}
  \crefname{corollary}{Corollary}{Corollaries}
  \crefname{section}{Section}{Sections}
  \crefname{figure}{Figure}{Figures}
    \crefname{assumption}{Assumption}{Assumptions}

\newtheorem{theorem}{Theorem}[section]
\newtheorem{theorem*}{Theorem*}[section]

\newtheorem{claim}[theorem]{Claim}
\newtheorem{lemma}[theorem]{Lemma}
\newtheorem{corollary}[theorem]{Corollary}
\newtheorem{proposition}[theorem]{Proposition}
\newtheorem{definition}[theorem]{Definition}

\numberwithin{equation}{section}

\theoremstyle{definition}
\newtheorem*{remark}{Remark}

\newcommand{\bbR}{\mathbb R}
\newcommand{\bbZ}{\mathbb Z}
\newcommand{\eps}{\varepsilon}
\newcommand{\LCD}{\operatorname{LCD}}
\newcommand{\QLCD}{\operatorname{QLCD}}
\newcommand{\dist}{\operatorname{dist}}
\newcommand{\supp}{\operatorname{supp}}
\newcommand{\Comp}{\operatorname{Comp}}
\newcommand{\Incomp}{\operatorname{Incomp}}

\newcommand{\HS}{\mathrm{HS}}
\newcommand{\rank}{\operatorname{rank}}

\begin{document}

\title{Rank and Invertibility of Dense Signed Random Regular Matrices}

\vspace{1cm}

\author{\textsc{Huaijin Liang, Yuxiao Shi, Kexin Yu, and Tingzhou Yu}}

\vspace{0.3cm}

\begin{abstract}
Let $A$ be the (non-symmetric) adjacency matrix of a uniformly random $d$-regular directed graph on $n$ vertices, and let $\Xi$ be independent of $A$ with i.i.d.\ Rademacher entries. Suppose that $\min(d,n-d)\ge\lambda n$ for some fixed $\lambda\in(0,1/2]$.
We show that there exists $c > 0$, depending only on $\lambda$, such that
\[
\mathbf P_{A,\Xi}\{\rank(A\circ\Xi)\le n-k\}\le e^{-c nk},\qquad1\le k\le n.
\]
As an ingredient in the proof of the rank bound, we use the
case $\kappa=0$ of the following quantitative smallest singular value
estimate:
\[
 \mathbf P_{A,\Xi}\{s_n(A\circ\Xi)\le\kappa\}
 \le C\kappa\sqrt n+e^{-c’ n},
 \qquad \kappa\ge0,
\]
where $C,c’>0$ depend only on $\lambda$.
\end{abstract}


\maketitle

\section{Introduction}\label{sec1}

For positive integers $d\le n$, let $\mathcal{M}_{n,d}$ be the set of all $n\times n$ $\{0,1\}$-matrices whose row and column sums are equal to $d$. Let $A$ be uniformly distributed on $\mathcal{M}_{n,d}$, and let $\Xi=(\xi_{ij})_{1\le i,j\le n}$ have i.i.d. Rademacher entries and be independent of $A$. We consider the matrix $B:=A\circ\Xi$, where $\circ$ denotes the Hadamard (Schur) product. The matrix $B$ is the signed adjacency matrix of a random $d$-regular directed graph (digraph) on $n$ labeled vertices allowing self-loops, where each vertex has in-degree and out-degree $d$.

In 2017, Cook \cite[Theorem 1.3]{MR3602844} studied the singularity probability of this model and proved that when $C\log^2 n\le d \le n$, \[ \mathbf P(\rank (B)<n)=O(d^{-1/4}).\]  
For $d=\lceil r n\rceil$ with fixed $r\in (0,1]$, Ferber, Jain, Luh, and Samotij \cite{ferber2021counting} obtained the exponential type probability bound, \[\mathbf P(\rank (B)<n)\le \exp(-n^{c})\] for some constant $c=c(r)>0$. However, both estimates are still far from the expected  exponentially small bound on the singularity probability.
We work in the dense regime $\min(d,n-d)\ge\lambda n$, where $\lambda \in (0, 1/2]$ is fixed. In this regime, we obtain an exponentially small singularity probability and a large deviation bound for the rank. More precisely, we ask how $\mathbf P(\rank (B)\le n-k)$ depends on $k$, allowing $k$ to grow throughout $1\le k \le n$. 

The corresponding questions for matrices with independent entries have a long history.
A central problem in combinatorial random matrix theory is to determine the singularity probability of an i.i.d. Rademacher matrix $\Xi$. It was conjectured that \[\mathbf P(\rank(\Xi)<n)=(2+o(1))n^22^{-n}.\] 
In a breakthrough work, Tikhomirov \cite{tikhomirov2020singularity} determined the sharp exponential rate \[\mathbf P(\rank(\Xi)<n)=(1/2+o(1))^n.\] Recently, Hu, Song, and Wu \cite{HuSongWu2026} resolved this conjecture.

Estimating the probability of larger corank is a very natural extension of the singularity problem. In 1995,  Kahn, Koml\'{o}s, and Szemer\'{e}di \cite{kahn1995probability} proved that for every $\varepsilon>0$, there exists a constant $C=C(\varepsilon)>0$ such that \[\mathbf P(\rank (\Xi)\le n-C)\le \varepsilon^n.\] Rudelson \cite{rudelson2024large} proved the bound of the form $\exp(-c_1kn)$  for matrices with i.i.d.  nonconstant subgaussian entries in the range $1\le k \le c_2\sqrt{n}$, where $c_1, c_2$ are positive constants. Hunter, Kwan,
Sauermann, and Sawhney established this bound for i.i.d. Rademacher entries over the full range $1\le k\le n$ in \cite{MR5088702}. Polavarapu extended this full range large corank estimate to independent but not necessarily  identically distributed real entries whose atom probabilities are uniformly bounded away from one \cite{polavarapu2026exponentialrankboundsrandom}.
The sharper exponential rates are known over smaller ranges of $k$. For an $n\times n$ matrix with i.i.d. $\{0,1\}$-valued Bernoulli entries of fixed mean $p\in (0,1/2]$, Jain, Sah, and Sawhney proved that the probability of corank at least $k$ is $(1-p+o(1))^{kn}$ for every fixed $k$ in \cite{jain2022rank}. Song and Wang recently extended this asymptotic rate to $k=o(\sqrt{\log n})$ in \cite{song2026rank}.


A related result holds for random symmetric matrices. For a symmetric matrix whose entries on and above the diagonal are i.i.d. copies of a mean zero, variance one subgaussian random variable, Han \cite{han2026rank} proved a bound of the form $\exp(-c_1kn)$ in the range $1\le k\le c_2\sqrt n$, where the positive constants $c_1, c_2$ depend only on the entry distribution.

Our first main result establishes a full range large deviation inequality for the signed adjacency matrix of a random $d$-regular directed graph in the dense regime, which matches the $kn$ exponent scale of the independent entry results in \cite{MR5088702, polavarapu2026exponentialrankboundsrandom}.






\begin{theorem}\label{thm:rank}
Fix $\lambda\in(0,1/2]$. There exists a constant $c_{\ref{thm:rank}}>0$, depending only on $\lambda$, such that the following holds. If $\min(d,n-d)\ge\lambda n$, 
then, for every $1\le k\le n$,
\[
        \mathbf P_{A,\Xi}\{\rank(A\circ\Xi)\le n-k\}
        \le e^{-c_{\ref{thm:rank}}kn}.
\]
\end{theorem}

\begin{remark}
Using the Bernoulli decomposition and relative anticoncentration
results of Polavarapu~\cite{polavarapu2026exponentialrankboundsrandom}, the argument can be extended to independent, not necessarily identically distributed, real weights whose largest atoms are uniformly bounded away from one. We restrict throughout to Rademacher signs in order to keep the statement and proof focused.
\end{remark}

Three of the present authors prove in a companion
manuscript~\cite{Unsigned} that the (unsigned) adjacency matrix $A$
satisfies
\[
 \mathbf P\{\rank(A)\le n-k\}
 \le 2e^{-c_\lambda nk},
 \qquad 1\le k\le n,
\]
in the same dense regime. The proof requires different methods since,
once the support is fixed, the unsigned model has no remaining
independent signs.

We also prove a quantitative lower bound estimate for the smallest singular value of $A\circ\Xi$ denoted by $s_n(A\circ\Xi)$, which is our second main result. For fixed $r\in (0,1)$ and $d=\lfloor rn \rfloor$, Cook \cite[Proposition 2.4]{cook2017circular} proved the following estimate for a more general matrix $A\circ X$, where $X$ has i.i.d. centered complex entries of unit variance and finite 4-th moment and is independent of $A$: for every fixed $z\in \mathbb{C}$ and every $t>0$,
\[
   \mathbf P_{A,\Xi}\{s_n(A\circ X-z\sqrt{nr}\mathbf{I}_n)\le t/\sqrt{n}\}
        \le C(t+n^{-1/2}),
\]
where $C$ is some positive constant depending on $r, z$ and $\mathbf{I}_n$ is the identity matrix.

For the unshifted matrix, our second main result strengthens Cook's result by replacing the polynomial bound $n^{-1/2}$ with an exponentially small term. In particular, as $t=0$, it improves the polynomial singularity probability bound to an exponentially small bound. This quantitative estimate is an independent result and it supplies the singularity probability bound used in the proof of Theorem~\ref{thm:rank}.

\begin{theorem}\label{thm:ssv}
Fix $\lambda\in(0,1/2]$. There exist constants $C_{\ref{thm:ssv}},c_{\ref{thm:ssv}}>0$, depending only on $\lambda$, such that the following holds. If  $\min(d,n-d)\ge\lambda n$, 
then, for every $\kappa\ge0$,
\[
        \mathbf P_{A,\Xi}\{s_n(A\circ\Xi)\le\kappa\}
        \le C_{\ref{thm:ssv}}\kappa\sqrt n
        +e^{-c_{\ref{thm:ssv}}n}.
\]
\end{theorem}

\begin{remark}
    The proof framework of Theorem~\ref{thm:ssv} can be
adapted, with minor modifications, to the case where the Rademacher
entries of $\Xi$ are replaced by i.i.d.\ centered, unit variance real
subgaussian random variables. The constants may then also depend on
the subgaussian norm of the entries. For clarity, we restrict
throughout to Rademacher signs.
\end{remark}




\subsection{Proof strategy}

We first consider the range $k\ge K\log n$ in Theorem~\ref{thm:rank}. Let $Z$ have independent entries which equal $1$ and $-1$ each with probability $d/(2n)$ and vanish otherwise. Conditioning every row and column of $|Z|$ to have sum $d$ gives the distribution of $A\circ\Xi$. This conditioning costs a factor of at most $\exp(C_\lambda n\log n)$. The rank estimate of Polavarapu~\cite{polavarapu2026exponentialrankboundsrandom} gives $\exp(-c_\lambda nk)$ before conditioning, so the required bound follows once $K$ is sufficiently large. This comparison does not retain the required exponent for $k<K\log n$.

For the remaining range, we work with the submatrices $B_I$ formed by the rows indexed by $I\subset[n]$. A sparse vector in the left kernel gives a linear dependence among a small number of rows. Deleting one of these rows preserves the rank and reduces the corank by one. Sparsity allows us to sum the bound from Odlyzko's lemma over the possible supports and retain a factor $e^{-cn}$ at each reduction. The sum of the number of deleted rows and the remaining target corank is preserved, so the induction remains within the required logarithmic range. At each stage, we use this reduction in the sparse case and a separate bound in the complementary case.

We next consider the complementary case, where the left kernel contains no nonzero sparse vector. Fix $A$ and denote the columns of $B_I$ by $X_1,\dots,X_n$. In the complementary case, we estimate the probability that these columns span a fixed subspace $V\subset\bbR^I$ of codimension $h$ whose orthogonal complement contains no nonzero sparse vector. On this event, $V^\perp=\ker(B_I^T)$ and $\operatorname{corank}(B_I)=h$.

The $j$-th column is random only on its support $C_j:=T_j^*(A)\cap I$, so the relevant codimension is $r_j(V):=\dim(P_{C_j}V^\perp)$, where $P_{C_j}$ denotes coordinate projection. Clearly $r_j(V)\le h$, since projection cannot increase dimension. A lower bound comparable to $h$ need not hold for every $j$. We obtain the required bounds by combining the absence of sparse vectors in $V^\perp$ with the covering property of the column supports.

The global codimension $h$ alone does not give a comparable lower bound for every $r_j(V)$. We obtain the required bounds by combining the absence of sparse vectors in $V^\perp$ with the covering property of the column supports.

The covering defect is motivated by the covering condition in the well-spread families of Jain, Sah, and Sawhney~\cite{MR4523249}. In our setting, this condition requires the union of any $Q$ column supports to cover all but at most $\eta n$ coordinates. The probability that this condition does not hold is at most $e^{-cn}$. This bound is sufficient for the smallest singular value estimate, but not for the rank bound $e^{-cnk}$ when $k$ grows. We therefore need to distinguish among supports that do not satisfy the covering condition.

For this purpose, we define $\operatorname{D}(A)$ as the largest number of pairwise disjoint sets of $Q$ column indices such that, for each set, the union of the corresponding supports leaves more than $\eta n$ coordinates uncovered. In particular, $\operatorname{D}(A)=0$ exactly when
\[
\left|[n]\setminus\bigcup_{j\in J}T_j^*(A)\right|\le\eta n\qquad\text{for every }J\subset[n]\text{ with }|J|=Q,
\]
which is the covering condition in Definition~\ref{def:well-spread}. More generally, $\operatorname{D}(A)$ bounds the number of columns with small effective codimension, while its tail satisfies $\mathbf P_A\{\operatorname{D}(A)\ge r\}\le e^{-cnr}$ for $1\le r\le K\log n$. These two properties allow us to choose a defect threshold depending on the target corank and retain the required exponent after averaging over $A$.

The absence of sparse vectors in $V^\perp$ gives the key bound
\[
\left|\left\{j\in[n]:r_j(V)<\frac hQ\right\}\right|<Q\bigl(\operatorname{D}(A)+1\bigr).
\]
Indeed, $Q$ coordinate restrictions of rank less than $h/Q$ have a common nonzero kernel vector, which would be sparse if the corresponding supports satisfied the covering condition. For sufficiently large $h$, the relative anticoncentration result of Hunter, Kwan, Sauermann, and Sawhney~\cite[Proposition~2.1]{MR5088702} applies on the remaining supports. A thin and thick decomposition, together with the covering defect tail, gives the required bound after averaging over $A$.

We complete the proof by induction, distinguishing at each stage whether the left kernel contains a nonzero sparse vector. In the sparse case, the resulting dependence among a small number of rows allows us to delete a redundant row, reduce the target corank by one, and obtain a factor $e^{-cn}$. In the complementary case, we use the preceding bound. The sum of the number of deleted rows and the remaining target corank is preserved, so the induction stays within the required logarithmic range. Theorem~\ref{thm:ssv} at $\kappa=0$ supplies the initial case and the bounds for bounded corank. Combining the two cases with the large corank bound proves Theorem~\ref{thm:rank}. In fact, the row deletion and covering defect arguments also prove the bound for $1\le k\le\sqrt n$ without using the large corank bound of Polavarapu~\cite{polavarapu2026exponentialrankboundsrandom}.
\subsection{Notation}\label{subsec:notation}

Throughout the paper, $\lambda\in(0,1/2]$ is fixed and $\min(d,n-d)\ge\lambda n$. For a positive integer $N$, set $[N]:=\{1,\dots,N\}$. The Euclidean norm and inner product are denoted by $\|\cdot\|_2$ and $\langle\cdot,\cdot\rangle$, respectively. We denote the unit sphere and unit ball in $\bbR^N$ by $\mathbb S^{N-1}$ and $B_2^N$. For $S\subset\bbR^N$, define $\dist(x,S):=\inf_{y\in S}\|x-y\|_2$.

For a vector $x\in\bbR^N$, its support is $\supp(x):=\{j\in[N]:x_j\ne0\}$. For $T\subset[N]$, the restriction of $x$ to $T$ is denoted by $x|_T$, and $P_T$ denotes the corresponding coordinate projection. For $\delta\in(0,1)$, define $\operatorname{Sparse}(\delta):=\{x\in\bbR^n:|\supp(x)|\le\delta n\}$. If $I\subset[n]$, we set $I^c:=[n]\setminus I$ and identify vectors in $\bbR^I$ with their extensions by zero on $I^c$ when regarded as elements of $\bbR^n$. In particular, sparsity for such vectors is measured relative to $n$, not $|I|$.

Unless otherwise stated, we denote the $i$-th row of a matrix $M$ by $M_i$ and its $j$-th column by $M^{(j)}$. Its Hilbert--Schmidt norm is $\|M\|_{\HS}:=(\sum_{i,j}|M_{ij}|^2)^{1/2}$.

The subscripts in $\mathbf P_A$, $\mathbf P_\Xi$, and $\mathbf P_{A,\Xi}$ specify the sources of randomness. We omit the subscript when the distribution is clear from the context. The same convention applies to expectations. The indicator of an event $\mathcal E$ is denoted by $\mathbf1_{\mathcal E}$. 

The constants $C,c>0$ may change from line to line. Their dependence on fixed parameters is indicated by subscripts or stated explicitly. Constants carrying a reference subscript are fixed in the corresponding result. The notation $f=O_\lambda(g)$ means that $|f|\le C_\lambda g$ for some constant $C_\lambda>0$. For nonnegative quantities, $f\lesssim_\lambda g$ means that $f\le C_\lambda g$; we use $\gtrsim_\lambda$ and $\asymp_\lambda$ analogously. For auxiliary parameters, $\varepsilon\ll\delta$ means that $\varepsilon$ is chosen sufficiently small in terms of $\delta$ and the parameters already fixed. All logarithms are natural unless indicated otherwise. Integer parts are omitted when they do not affect the estimates.

\subsection{Organization}

Section~\ref{sec:prelim} contains the local comparison for regular supports, the well-spread property, and the covering defect bound. In Section~\ref{sec:support-relative}, we relate the covering defect to the effective codimensions and prove the relative anticoncentration result on column supports. Section~\ref{sec:rank} proves Theorem~\ref{thm:rank} by combining the large corank bound, the sparse kernel reduction, and the thin and thick decomposition. Finally, Section~\ref{sec:ssv} proves Theorem~\ref{thm:ssv} independently of the rank argument. Its specialization at $\kappa=0$ provides the singularity bound used in Section~\ref{sec:rank}.

\section{Preliminaries}\label{sec:prelim}
For $A\in \mathcal{M}_{n,d}$, define its row and column supports by
\[
        T_i(A):=\supp(A_i),
        \qquad
        T_j^*(A):=\supp(A^{(j)}),
        \qquad i,j\in[n].
\]
We set $\mathcal T(A)=(T_i(A))_{i\in[n]}$ and $\mathcal T^*(A)=(T_j^*(A))_{j\in[n]}$.

The following comparison will be used both for the quasirandomness estimate and for the covering defect argument in the proof.

\begin{lemma}\label{lem:local-comparison}
Fix $K>0$. There exists $C_{\ref{lem:local-comparison}}>0$, depending only on $\lambda$ and $K$, such that the following holds. Let $1\le m\le K\log n$, and let $\mathcal E$ be an event depending on $m$ prescribed rows of $A$. Let $p_{\rm reg}$ be its probability under the uniform measure on $\mathcal{M}_{n,d}$, and let $p_{\rm Ber}$ be its probability when all entries in the prescribed rows are independent $\operatorname{Ber}(d/n)$ variables. Then
\[
        p_{\rm reg}
        \le
        \exp\bigl(C_{\ref{lem:local-comparison}}(m^2+m\log n)\bigr)
        p_{\rm Ber}.
\]
The same bound holds for events depending on prescribed columns.
\end{lemma}

\begin{proof}
A special case of the dense enumeration theorem of Canfield, Greenhill, and McKay~\cite[Theorem~1]{MR2378856} gives the following. Let $N_{\mathbf c}$ (respectively $N_{\mathbf c'}$) denote the number of $(n-m)\times n$ $\{0,1\}$-matrices with all row sums equal to $d$ and column sums $c_1,\dots,c_n$ (respectively $c'_1,\dots,c'_n$), where $c_j,c'_j\in[d-m,d]$. Then, for $m\le K\log n$,
\[
        \max\left\{\frac{N_{\mathbf c}}{N_{\mathbf c'}},\frac{N_{\mathbf c'}}{N_{\mathbf c}}\right\}
        \le
        e^{C_\lambda m^2}.
\]

Let $p_{\rm row}$ denote the probability of $\mathcal E$ when the prescribed rows are sampled independently and uniformly from the vectors in $\{0,1\}^n$ with exactly $d$ ones. Since this model assigns equal probability to every row configuration, the preceding comparison of completion numbers gives
\[
p_{\rm reg}\le e^{C_\lambda m^2}p_{\rm row}.
\]
Set $\tilde{\lambda}:=d/n$. Conditioning the $m$ prescribed Bernoulli rows to have sum $d$ gives the independent row model. By Stirling's formula, $\mathbf P\{\operatorname{Bin}(n,\tilde{\lambda})=d\}\ge(C_\lambda\sqrt n)^{-1}$ for all sufficiently large $n$. Hence
\[
p_{\rm row}\le\mathbf P\{\operatorname{Bin}(n,\tilde{\lambda})=d\}^{-m}p_{\rm Ber}\le(C_\lambda\sqrt n)^m p_{\rm Ber}.
\]
Combining the two bounds and choosing $C_{\ref{lem:local-comparison}}$ sufficiently large, we obtain
\[
p_{\rm reg}\le e^{C_\lambda m^2}(C_\lambda\sqrt n)^m p_{\rm Ber}\le\exp\bigl(C_{\ref{lem:local-comparison}}(m^2+m\log n)\bigr)p_{\rm Ber}.
\]
The column statement follows by transposition.
\end{proof}

\begin{definition}[Well-spread family]\label{def:well-spread}
Let $\mathcal T=\{\{T_1,\dots,T_n\}\}$ be an indexed family of subsets of $[n]$. We say that $\mathcal T$ is $(Q,\eta,\theta)$-well-spread if
\begin{enumerate}
\item $|T_i|\ge\theta n$ for every $i\in[n]$;
\item $|T_i\cap T_j|\ge\eta n$ for all distinct $i,j\in[n]$;
\item for every collection of distinct indices $i_1,\dots,i_Q$,
\[
        \left|[n]\setminus\bigcup_{r=1}^Q T_{i_r}\right|
        \le\eta n.
\]
\end{enumerate}
\end{definition}

The deterministic condition above is typical for dense random regular supports.

\begin{proposition}\label{prop:support}
Fix $\lambda\in(0,1/2]$ and $0<\eta\le\lambda^2/16$. There exist $Q=Q(\lambda,\eta)\in\mathbb N$ and $c_{\ref{prop:support}}>0$ such that, with $\theta:=\lambda$,
\[
\mathbf P_A\left\{
        \mathcal T(A)\text{ and }\mathcal T^*(A)
        \text{ are }(Q,\eta,\theta)\text{-well-spread}
\right\}
\ge
1-e^{-c_{\ref{prop:support}}n}.
\]
\end{proposition}

\begin{proof}
Set $\tilde{\lambda}:=d/n$. The first condition is deterministic. For the second, fix $i\ne j$. In the Bernoulli model, if $U_i,U_j$ denote the two row supports, then $|U_i\cap U_j|\sim\operatorname{Bin}(n,\tilde{\lambda}^2)$. 
Since $\eta\le\lambda^2/16$, Chernoff's inequality gives
\[
        \mathbf P_{\rm Ber}\{|U_i\cap U_j|\le\eta n\}
        \le e^{-cn}.
\]
Lemma~\ref{lem:local-comparison}, followed by a union bound over the pairs of rows, gives an exponentially small failure probability.

Choose $Q$ sufficiently large so that 
        $\lambda Q\ge8\log(e/\eta)$. 
Fix distinct $i_1,\dots,i_Q$. If
\[
        \left|[n]\setminus\bigcup_{r=1}^QU_{i_r}\right|
        \ge\eta n,
\]
then there exists $J\subset[n]$ with $|J|=\lceil\eta n\rceil$ on which all $Q$ Bernoulli rows vanish. Hence
\begin{align*}
&\mathbf P_{\rm Ber}\left\{
        \left|[n]\setminus\bigcup_{r=1}^QU_{i_r}\right|
        \ge\eta n
\right\} \\
&\qquad\le
\binom{n}{\lceil\eta n\rceil}(1-\tilde{\lambda})^{Q\lceil\eta n\rceil}
\le
\exp\left(\eta n\log(e/\eta)-\lambda Q\eta n+O_\eta(\log n)\right)
\le e^{-cn},
\end{align*}
where the last inequality follows from $\lambda Q\ge8\log(e/\eta)$. Lemma~\ref{lem:local-comparison} and a union bound over the at most $n^Q$ choices of the row indices still leave an exponentially small failure probability. Applying the same argument to $A^T$ completes the proof.
\end{proof}
For the smallest singular value estimate, Proposition~\ref{prop:support} allows us to restrict to supports $A\in \mathcal{M}_{n,d}$ for which both $\mathcal T(A)$ and $\mathcal T^*(A)$ are well-spread except on an event of probability at most $e^{-cn}$. We can then fix any such support and prove the required bound with respect to the signs. This bound is sufficient for Theorem~\ref{thm:ssv}, but is too large for the rank bound $e^{-cnk}$ in Theorem~\ref{thm:rank} when $k$ grows. Therefore, we need a more precise decomposition of the supports that retains the dependence on $k$.

The following covering defect refines the covering property in Definition~\ref{def:well-spread} by recording the maximal number of pairwise disjoint failures.
\begin{definition}[Covering defect]\label{def:defect}
A set $J\subset[n]$ with $|J|=Q$ is called a $(Q,\eta)$-cover-deficient set if
\[
        \left|
        [n]\setminus
        \bigcup_{j\in J}T_j^*(A)
        \right|
        >
        \eta n.
\]
The covering defect of $A$, denoted by $\operatorname{D}(A)$, is the largest integer $r\ge0$ for which there exist $r$ pairwise disjoint
$(Q,\eta)$-cover-deficient sets.
\end{definition}

The next lemma gives an exponential tail bound for the covering defect.
This estimate retains the $nk$ scale after averaging over the random
support.

\begin{lemma}\label{lem:defect-tail}
For every fixed $K>0$, there exists
$c_{\ref{lem:defect-tail}}>0$, depending only on $\lambda,\eta,Q$, and $K$, such that, for every $1\le r\le K\log n$, 
one has
\[
        \mathbf P_A\{\operatorname{D}(A)\ge r\}
        \le
        e^{-c_{\ref{lem:defect-tail}}nr}.
\]
\end{lemma}

\begin{proof}
Suppose that $\operatorname{D}(A)\ge r$. Then there exist pairwise
disjoint sets $J_1,\dots,J_r\subset[n]$, $|J_\ell|=Q$, such that every $J_\ell$ is $(Q,\eta)$-cover-deficient. Hence, for each $\ell\in[r]$, there exists a set $R_\ell\subset[n]$, $|R_\ell|=\lceil\eta n\rceil$, such that $A_{ij}=0$ for every $i\in R_\ell$ and $j\in J_\ell$. 

Fix the sets $J_1,\dots,J_r$ and $R_1,\dots,R_r$. In the Bernoulli
model, the corresponding event has probability
\[
\mathbf P_{\rm Ber}
\left\{
        A_{ij}=0:
        i\in R_\ell,\ j\in J_\ell,\ \ell\in[r]
\right\}=
        (1-\frac{d}{n})^{Qr\lceil\eta n\rceil}\le
        e^{-\lambda Q\eta nr}.
\]
The number of possible choices of the column sets and the uncovered row
sets is at most
\[
\begin{aligned}
n^{Qr}
\binom{n}{\lceil\eta n\rceil}^{r}
&\le
\exp\left(
        Qr\log n
        +
        \eta nr\log(e/\eta)
        +
        O_\eta(r\log n)
\right).
\end{aligned}
\]
For each fixed choice, the event depends on the $Qr$ columns indexed by $J_1\cup\cdots\cup J_r$. 
Since $Qr\le QK\log n$, the column version of Lemma~\ref{lem:local-comparison}, with the constant $K$ there replaced by $QK$, gives a comparison factor of at most
\[
\begin{aligned}
\exp\left(
        C_{\ref{lem:local-comparison}}
        \bigl((Qr)^2+Qr\log n\bigr)
\right)
&=
\exp\left(
        O_{\lambda,Q,K}\bigl((\log n)^2\bigr)
\right).
\end{aligned}
\]
Taking the union bound over all choices, we obtain
\[
\log
\mathbf P_A\{\operatorname{D}(A)\ge r\}
\le -\lambda Q\eta nr + \eta nr\log(e/\eta) + Qr\log n + O_\eta(r\log n)+ O_{\lambda,Q,K}\bigl((\log n)^2\bigr). 
\]
Choose $Q$ sufficiently large so that $\lambda Q\ge8\log(e/\eta)$. It follows that, for all sufficiently large $n$,
\[
        \mathbf P_A\{\operatorname{D}(A)\ge r\}
        \le
        e^{-c_{\ref{lem:defect-tail}}nr}
\]
for some
$c_{\ref{lem:defect-tail}}>0$ depending only on
$\lambda,\eta,Q$, and $K$.
\end{proof}

\section{Relative anticoncentration on regular supports}\label{sec:support-relative}

We first record the high-dimensional relative anticoncentration result of Hunter, Kwan, Sauermann, and Sawhney \cite[Proposition~2.1]{MR5088702}. Fix $p\in(0,1/32]$, and let $\nu_p$ be the law on $\{-1,0,1\}$ given by $\nu_p(1)=\nu_p(-1)=p$, 
        $\nu_p(0)=1-2p$. 
\begin{lemma}\label{lem:hunter-relative}
There exist $\gamma_{\ref{lem:hunter-relative}}\in(0,1)$ and $k_{\ref{lem:hunter-relative}}\in\mathbb N$ such that the following holds. Let $r\ge k_{\ref{lem:hunter-relative}}$, let $V\subset \mathbb Q^m$ be a linear subspace of codimension $r$, let $X$ be uniform on $\{\pm1\}^m$, and let $Y\sim\nu_p^{\otimes m}$. Then
\[
        \mathbf P\{X\in V\}
        \le
        \gamma_{\ref{lem:hunter-relative}}^r
        \mathbf P\{Y\in V\}.
\]
\end{lemma}
We also use the following standard bound for the comparison vectors.
\begin{lemma}\label{lem:odlyzko}
Let $Z=(Z_1,\dots,Z_m)$ have independent coordinates satisfying
\[
        \sup_{u\in\bbR}\mathbf P\{Z_j=u\}\le q<1
        \qquad(j\in[m]).
\]
If $W\subset\bbR^m$ is an affine subspace of codimension $r$, then
\[
        \mathbf P\{Z\in W\}\le q^r.
\]
\end{lemma}

\begin{proof}
Set $W=\{z:Lz=b\}$, where $L$ has rank $r$. After permuting the coordinates, the first $r$ columns of $L$ are linearly independent. Conditional on the remaining coordinates, membership in $W$ determines the first $r$ coordinates uniquely, and the result follows from independence.
\end{proof}

The following notation will be used throughout this section and the next. Fix $A\in \mathcal{M}_{n,d}$ and $I\subset[n]$, and denote by $B_I$ the submatrix
of $B$ formed by the rows indexed by $I$. For $j\in[n]$, set
\[
        C_j:=T_j^*(A)\cap I,
        \qquad
        E_j
        :=
        \left\{
        x\in\bbR^I:
        \supp(x)\subset C_j
        \right\},
\]
and set $X_j\in E_j$ for the $j$-th column of $B_I$.

Independently of $\Xi$, take $Y_1,\dots,Y_n$ to be independent comparison vectors satisfying $Y_j\in E_j$ and $Y_j|_{C_j}\sim\nu_p^{\otimes |C_j|}$ for every $j\in[n]$.
 For a subspace $V\subset\mathbb R^I$, define
\[
        r_j(V)
        :=
        \operatorname{codim}_{E_j}(V\cap E_j)
        =
        \dim(P_{C_j}V^\perp),
\]
where $P_{C_j}$ denotes the coordinate projection onto
$\bbR^{C_j}$. 

The next proposition is the main result of this section. It establishes the high-dimensional relative anticoncentration result for the column supports of the regular model.

\begin{proposition}\label{prop:relative-regular-support}
Let $A\in \mathcal{M}_{n,d}$ and $I\subset[n]$ be fixed. Let $V\subset\bbR^I$ be a rational subspace of codimension $h$, and suppose that $V^\perp$ contains no nonzero vector supported on at most $\delta n$ coordinates. If $h\ge Qk_{\ref{lem:hunter-relative}}$, then, for all but fewer than $Q\bigl(\operatorname{D}(A)+1\bigr)$ indices $j\in[n]$,
\[
        \mathbf P_\Xi\{X_j\in V\}
        \le
        \gamma_{\ref{lem:hunter-relative}}^{h/Q}
        \mathbf P\{Y_j\in V\}.
\]
\end{proposition}

The proof consists of two steps. The first establishes the relative
anticoncentration result on the support of the $j$-th column under the
condition $r_j(V)\ge k_{\ref{lem:hunter-relative}}$. The second shows that $r_j(V)\ge h/Q$
for all but fewer than $Q\bigl(\operatorname{D}(A)+1\bigr)$ columns.

\begin{lemma}\label{lem:column-relative}
Let $A\in \mathcal{M}_{n,d}$ and $I\subset[n]$ be fixed. Let $V\subset\bbR^I$ be a rational subspace. If $r_j(V)\ge k_{\ref{lem:hunter-relative}}$,
then
\[
        \mathbf P_\Xi\{X_j\in V\}
        \le
        \gamma_{\ref{lem:hunter-relative}}^{r_j(V)}
        \mathbf P\{Y_j\in V\}.
\]
\end{lemma}

\begin{proof}
For $z=(z_i)_{i\in C_j}\in\bbR^{C_j}$, denote by
$\widetilde z\in\bbR^I$ its extension by zero outside $C_j$, and set
\[
        U_j
        :=
        \left\{
        z\in\bbR^{C_j}:
        \widetilde z\in V
        \right\}.
\]
Extending vectors by zero outside $C_j$ identifies $U_j$ with
$V\cap E_j$. Hence,
\[
        \operatorname{codim}_{\bbR^{C_j}}U_j
        =
        \operatorname{codim}_{E_j}(V\cap E_j)
        =
        r_j(V).
\]
Moreover,
\[
        \mathbf P_\Xi\{X_j\in V\}
        =
        \mathbf P_\Xi
        \left\{
        X_j|_{C_j}\in U_j
        \right\},
        \qquad
        \mathbf P\{Y_j\in V\}
        =
        \mathbf P
        \left\{
        Y_j|_{C_j}\in U_j
        \right\}.
\]
Now $X_j|_{C_j}$ is uniform on $\{\pm1\}^{C_j}$ and $Y_j|_{C_j}\sim\nu_p^{\otimes |C_j|}$. 
Lemma~\ref{lem:hunter-relative}, applied in the
subspace $U_j$, gives
\[
        \mathbf P_\Xi\{X_j\in V\}
        \le
        \gamma_{\ref{lem:hunter-relative}}^{r_j(V)}
        \mathbf P\{Y_j\in V\}.
\]
\end{proof}

Lemma~\ref{lem:column-relative} shows that the relevant exponent is the local codimension $r_j(V)$ rather than the global codimension $h$. The next lemma shows that the two quantities are comparable for all but a set of columns whose cardinality is bounded by the covering defect.

\begin{lemma}\label{lem:effective-codimension}
Let $I\subset[n]$, let $V\subset\bbR^I$ be a rational subspace of codimension $h$, and suppose that $V^\perp$ contains no nonzero vector supported on at most $\delta n$ coordinates. Then
\[
        \left|
        \left\{
        j\in[n]:
        r_j(V)<\frac hQ
        \right\}
        \right|
        <
        Q\bigl(\operatorname{D}(A)+1\bigr).
\]
\end{lemma}

\begin{proof}
Let
\[
        \mathcal J
        :=
        \left\{
        j\in[n]:
        r_j(V)<\frac hQ
        \right\}.
\]
Suppose that $J=\{j_1,\dots,j_Q\}\subset\mathcal J$ is not a $(Q,\eta)$-cover-deficient set. Consider
\[
        \Phi:
        V^\perp
        \longrightarrow
        \bigoplus_{\ell=1}^Q\bbR^{C_{j_\ell}},
        \qquad
        \Phi(w)
        :=
        \bigl(
        w|_{C_{j_1}},\dots,w|_{C_{j_Q}}
        \bigr).
\]
For every $\ell\in[Q]$, the image of the $\ell$-th component map is
$P_{C_{j_\ell}}V^\perp$. Hence,
\[
        \rank(\Phi)\le
        \sum_{\ell=1}^Q
        \dim(P_{C_{j_\ell}}V^\perp)=
        \sum_{\ell=1}^Q r_{j_\ell}(V)<
        h=\dim(V^\perp).
\]
There exists a nonzero $w\in\ker\Phi$. This vector vanishes on every
$C_{j_\ell}$, and hence
\[
        \supp(w)
        \subset
        I\setminus
        \bigcup_{\ell=1}^Q C_{j_\ell}
        \subset
        [n]\setminus
        \bigcup_{\ell=1}^Q T_{j_\ell}^*(A).
\]
Since $J$ is not $(Q,\eta)$-cover-deficient, the last set has
cardinality at most $\eta n<\delta n$. 
This contradicts the assumption on $V^\perp$.

Thus every $Q$-subset of $\mathcal J$ is
$(Q,\eta)$-cover-deficient. If $|\mathcal J| \ge Q\bigl(\operatorname{D}(A)+1\bigr)$,
then $\mathcal J$ contains
$\operatorname{D}(A)+1$ pairwise disjoint
$(Q,\eta)$-cover-deficient sets, contrary to the definition of
$\operatorname{D}(A)$.
\end{proof}

\begin{proof}[Proof of Proposition~\ref{prop:relative-regular-support}]
Lemma~\ref{lem:effective-codimension} shows that $r_j(V)\ge h/Q$ 
for all but fewer than $Q\bigl(\operatorname{D}(A)+1\bigr)$ indices $j\in[n]$. For every such $j$, by the assumption, we have
        $r_j(V)
        \ge
          h/Q
        \ge
        k_{\ref{lem:hunter-relative}}$. 
Lemma~\ref{lem:column-relative} therefore gives
\[
        \mathbf P_\Xi\{X_j\in V\}
        \le
        \gamma_{\ref{lem:hunter-relative}}^{r_j(V)}
        \mathbf P\{Y_j\in V\}.
\]
Since $0<\gamma_{\ref{lem:hunter-relative}}<1$ and $r_j(V)\ge h/Q$, it follows that
\[
        \mathbf P_\Xi\{X_j\in V\}
        \le
        \gamma_{\ref{lem:hunter-relative}}^{h/Q}
        \mathbf P\{Y_j\in V\}.
\]
\end{proof}

\section{Proof of Theorem~\ref{thm:rank}}\label{sec:rank}
The proof is divided according to whether the corank is above or below a logarithmic threshold. We first handle the large corank range. For the remaining range, we introduce matrices obtained by deleting rows and recursively remove sparse vectors from the left kernel. When the left kernel contains no such vector, we use a thin and thick decomposition
adapted to the support.

\subsection{The large corank regime}
We first consider the regime in which the corank exceeds a logarithmic
threshold. 
\begin{proposition}\label{lem:large-corank}
There exist
$K_{\ref{lem:large-corank}},c_{\ref{lem:large-corank}}>0$, depending only
on $\lambda$, such that, for every $k\ge K_{\ref{lem:large-corank}}\log n$,
one has
\[
        \mathbf P_{A,\Xi}
        \left\{
        \operatorname{corank}(A\circ\Xi)\ge k
        \right\}
        \le
        e^{-c_{\ref{lem:large-corank}}nk}.
\] 
\end{proposition}
In this range, we compare the signed regular model with a matrix having independent entries and recover the regular support by conditioning.
Set $\tilde{\lambda}:= d/n$, and let $Z=(Z_{ij})_{i,j\in[n]}$ have independent entries satisfying $\mathbf P\{Z_{ij}=0\}=1-\tilde{\lambda}$ and $ \mathbf P\{Z_{ij}=1\} = \mathbf P\{Z_{ij}=-1\} = {\tilde{\lambda}}/{2}$. 

Let $\mathcal R_{n,d}$ denote the event that all row and column sums of $|Z|=(|Z_{ij}|)_{i,j\in[n]}$ are equal to $d$. The next lemma estimates the probability of this event and identifies the conditional distribution of $Z$.

\begin{lemma}\label{lem:regular-conditioning}
There exists
$C_{\ref{lem:regular-conditioning}}>0$, depending only on $\lambda$, such
that
\[
        \mathbf P(\mathcal R_{n,d})
        \ge
        e^{-C_{\ref{lem:regular-conditioning}}n\log n}.
\]
Conditional on $\mathcal R_{n,d}$, the matrix $Z$ has the same distribution
as $A\circ\Xi$.
\end{lemma}

\begin{proof}
For every $A_0\in \mathcal{M}_{n,d}$, 
        $\mathbf P\{|Z|=A_0\}
        =
        \tilde{\lambda}^{nd}(1-\tilde{\lambda})^{n^2-nd}$, 
which does not depend on $A_0$. Since
\[
        \mathcal R_{n,d}
        =
        \bigsqcup_{A_0\in \mathcal{M}_{n,d}}
        \{|Z|=A_0\},
\]
we have $\mathbf P(\mathcal R_{n,d})
        =
        |\mathcal{M}_{n,d}|
        \tilde{\lambda}^{nd}(1-\tilde{\lambda})^{n^2-nd}$. 
The dense enumeration theorem of Canfield, Greenhill, and McKay~\cite[Theorem~1]{MR2378856}, together with Stirling's formula, gives $\log\mathbf P(\mathcal R_{n,d})
        =
        -n\log n+O_\lambda(n)$. 
This proves the required lower bound.

It remains to identify the conditional distribution. Fix $A_0\in \mathcal{M}_{n,d}$ and a sign configuration
\[
        \varepsilon
        =
        (\varepsilon_{ij})_{(A_{0})_{ij}=1}
        \in\{\pm1\}^{nd}
\]
on the support of $A_0$. Then
\[
\mathbf P\left\{
        |Z|=A_0,\
        (Z_{ij})_{(A_{0})_{ij}=1}=\varepsilon
        \right\}=
        \left(\frac{\tilde{\lambda}}{2}\right)^{nd}
        (1-\tilde{\lambda})^{n^2-nd},
\]
which does not depend on either $A_0$ or $\varepsilon$. Thus,
\[
\mathbf P\left\{
        |Z|=A_0,\
        (Z_{ij})_{(A_{0})_{ij}=1}=\varepsilon
        \,\middle|\,
        \mathcal R_{n,d}
\right\}
=
        \frac{1}{|\mathcal{M}_{n,d}|\,2^{nd}}.
\]
Thus, conditional on $\mathcal R_{n,d}$, the support is uniform on
$\mathcal{M}_{n,d}$ and, given the support, the signs on its nonzero entries are
independent Rademacher variables. Therefore, $Z\mid\mathcal R_{n,d}
        \stackrel{d}{=}
        A\circ\Xi.$ 
\end{proof}

We also use the following rank estimate for matrices with independent,
not necessarily identically distributed, entries.

\begin{lemma}[{\cite[Theorem~1.1]{polavarapu2026exponentialrankboundsrandom}}]\label{lem:polavarapu}
Fix $b\in(0,1)$. There exists
$c_{\ref{lem:polavarapu}}=c_{\ref{lem:polavarapu}}(b)>0$ such that, if $M$ is an $n\times n$
random matrix with independent real entries satisfying
\[
        \sup_{u\in\bbR}
        \mathbf P\{M_{ij}=u\}
        \le b
        \qquad i,j\in[n],
\]
then, for every $1\le k\le n$,
\[
        \mathbf P\{\rank(M)\le n-k\}
        \le
        e^{-c_{\ref{lem:polavarapu}}nk}.
\]
\end{lemma}

\begin{proof}[Proof of Proposition~\ref{lem:large-corank}]
The largest atom of an entry of $Z$ is at most $b_\lambda
        :=
        \max\{1-\lambda,1/2\}
        <1$. 
Lemma~\ref{lem:polavarapu} gives
\[
        \mathbf P\{\rank(Z)\le n-k\}
        \le
        e^{-c_{\ref{lem:polavarapu}}nk}.
\]
By Lemma~\ref{lem:regular-conditioning}, 
\begin{align*}
\mathbf P_{A,\Xi}
\left\{\rank(A\circ\Xi)\le n-k\right\}&=
\mathbf P_{A,\Xi}
\left\{\rank(Z)\le n-k\mid\mathcal R_{n,d}\right\}\\
&\le
\exp\left(
        -c_{\ref{lem:polavarapu}}nk
        +
        C_{\ref{lem:regular-conditioning}}n\log n
\right).
\end{align*} 
Choosing $K_{\ref{lem:large-corank}}$ sufficiently large completes the
proof.
\end{proof}

\subsection{Sparse vectors in the left kernel}

We now consider the remaining range. Sparse vectors in the left kernel
are removed recursively by deleting redundant rows. This requires a
formulation for matrices obtained after deleting rows.

Set $B:=A\circ\Xi$. For $I\subset[n]$, denote by $B_I$ the submatrix of $B$ formed by the rows indexed by $I$. Throughout the remainder of the paper, for a matrix $M$ with $m$ rows, define
\[
        \operatorname{corank}(M):=m-\rank(M)=\dim\ker(M^T).
\]
In particular, $\operatorname{corank}(B_I)
        =
        |I|-\rank(B_I)
        =
        \dim\ker(B_I^T)$. 
This notion of corank is natural for the row deletion argument. If row $i$ belongs to the span of the remaining rows, then deleting it leaves the rank unchanged and 
        $\operatorname{corank}
        (B_{I\setminus\{i\}})
        =
        \operatorname{corank}(B_I)-1$. 

For $k\ge1$, define \[\mathcal E_k(I):=\left\{\operatorname{corank}(B_I)\ge k\right\},\]
and for $\delta \in (0,1)$ define
\[
 \mathcal S_I:= \left\{\ker(B_I^T)\cap \mathrm{Sparse}(\delta)  \neq \{0\}\right\}.
\]

The next lemma controls the sparse part of the induction.

\begin{lemma}\label{lem:sparse-kernel-reduction}
Let $\delta$ be sufficiently small in terms of $\lambda$. There exists
$c_{\ref{lem:sparse-kernel-reduction}}>0$, depending only on
$\lambda$ and $\delta$, such that, for every $I\subset[n]$ and every
$2\le k\le |I|$,
\[\mathbf P_{A,\Xi}
\left\{
        \mathcal E_k(I)\cap\mathcal S_I
\right\}\le
        e^{-c_{\ref{lem:sparse-kernel-reduction}}n}
        \sup_{i\in I}
        \mathbf P_{A,\Xi}
        \left\{
        \mathcal E_{k-1}(I\setminus\{i\})
        \right\}.
\]
\end{lemma}

\begin{proof}
For each $S\subset I$ with $1\le |S|\le\delta n$, set $i(S):=\min S$. 
Suppose that $\mathcal E_k(I)\cap\mathcal S_I$ occurs. Then there is a
nonzero vector $u\in\ker(B_I^T)$ such that $S:=\supp(u)$, $1\le |S|\le\delta n$.
Since $\sum_{j\in S}u_jB_j=0$ and $u_{i(S)}\ne0$, one has $B_{i(S)}
        \in
        \operatorname{span}
        \{B_j:j\in S\setminus\{i(S)\}\}$. 
Deleting row $i(S)$ does not change the rank. Hence
\[
        \operatorname{corank}
        (B_{I\setminus\{i(S)\}})
        =
        \operatorname{corank}(B_I)-1
        \ge k-1.
\]
It follows that
\[
\mathcal E_k(I)\cap\mathcal S_I
\subset
\bigcup_{\substack{S\subset I\\1\le |S|\le\delta n}}
\left\{
B_{i(S)}
\in
\operatorname{span}
\{B_j:j\in S\setminus\{i(S)\}\},
\mathcal E_{k-1}(I\setminus\{i(S)\})
\right\}.
\]

Fix such a set $S$, and set $i:=i(S)$, $s:=|S|$. Condition on $A$ and the rows $(B_j)_{j\ne i}$. Under this condition, $W_{S,i}:=\operatorname{span}\{B_j:j\in S\setminus\{i\}\}$
is fixed, while $B_i|_{T_i(A)}$ is uniform on $\{\pm1\}^d$.

For $z\in\bbR^{T_i(A)}$, denote by $\widetilde z\in\bbR^n$ its extension
by zero outside $T_i(A)$, and define
\[
        U_{S,i}
        :=
        \left\{
        z\in\bbR^{T_i(A)}:
        \widetilde z\in W_{S,i}
        \right\}.
\]
Extending vectors by zero outside $T_i(A)$ gives a linear isomorphism from
$U_{S,i}$ onto 
\[W_{S,i}\cap
        \left\{
        x\in\bbR^n:
        \supp(x)\subset T_i(A)
        \right\}.\] 
Hence $\dim U_{S,i}\le\dim W_{S,i}\le s-1$. Since $|T_i(A)|=d$, Lemma~\ref{lem:odlyzko} applied with $q=1/2$ gives
\[
\mathbf P_\Xi\left\{B_i\in W_{S,i}\,\middle|\,A,(B_j)_{j\ne i}\right\}
=
\mathbf P_\Xi\left\{B_i|_{T_i(A)}\in U_{S,i}\,\middle|\,A,(B_j)_{j\ne i}\right\}
\le 2^{-(d-s+1)}.
\]
The event $\mathcal E_{k-1}(I\setminus\{i\})$ is measurable under the same conditions. Hence
\[
\begin{aligned}
\mathbf P_{A,\Xi}
\left\{
        B_i\in W_{S,i},\
        \mathcal E_{k-1}(I\setminus\{i\})
\right\}&=
\mathbf E_{A,\Xi}\left[
        \mathbf 1_{\mathcal E_{k-1}(I\setminus\{i\})}
        \mathbf P_\Xi
        \left\{
        B_i\in W_{S,i}
        \,\middle|\,
        A,(B_j)_{j\ne i}
        \right\}
\right]\\
&\le
        2^{-(d-s+1)}
        \mathbf P_{A,\Xi}
        \left\{
        \mathcal E_{k-1}(I\setminus\{i\})
        \right\}.
\end{aligned}
\]
Taking the union over $S$ yields
\[
\mathbf P_{A,\Xi}
\left\{
        \mathcal E_k(I)\cap\mathcal S_I
\right\}\le
\left(
\sum_{s=1}^{\lfloor\delta n\rfloor}
        \binom{|I|}{s}2^{-(d-s+1)}
\right)
\sup_{i\in I}
\mathbf P_{A,\Xi}
\left\{
        \mathcal E_{k-1}(I\setminus\{i\})
\right\}.
\]
Choose $\delta\in(0,\lambda)$ sufficiently small so that $H(\delta)<(\lambda-\delta)\log2$, where $H(t):=-t\log t-(1-t)\log(1-t)$. 
Since $d\ge\lambda n$,
\[
\sum_{s=1}^{\lfloor\delta n\rfloor}
        \binom{|I|}{s}2^{-(d-s+1)}
\le
\sum_{s=1}^{\lfloor\delta n\rfloor}
        \binom ns2^{-(d-s+1)}
\le
\exp\left(
        \bigl[
        H(\delta)-(\lambda-\delta)\log2+o(1)
        \bigr]n
\right)
\le
        e^{-c_{\ref{lem:sparse-kernel-reduction}}n}
\]
for all sufficiently large $n$.

\end{proof}

Each application of Lemma~\ref{lem:sparse-kernel-reduction} gives
\[
        r_0=n-|I|
        \longmapsto
        r_0+1,
        \qquad
        k
        \longmapsto
        k-1.
\]
Thus, $r_0+k$ is preserved throughout the induction. Since Proposition~\ref{lem:large-corank} gives the required estimate above the logarithmic threshold, it remains to consider pairs $(I,k)$ satisfying $r_0+k\le K\log n$. 
Next, we consider the event $\mathcal S_I^c$, where the left kernel
contains no sparse vector.

\subsection{A thin and thick decomposition}
Recall that $\mathcal S_I^c$ is the event that $\ker(B_I^T)$ contains no nonzero vector supported on at most $\delta n$ coordinates. We now estimate the exact corank on $\mathcal S_I^c$. The relative anticoncentration result from the preceding section gives the required gain for all but a number of columns bounded in terms of the covering defect.

\begin{proposition}\label{prop:fixed-defect}
There exist
$h_{\ref{prop:fixed-defect}}\in\mathbb N$ and
$c_{\ref{prop:fixed-defect}}>0$, depending only on the fixed parameters,
such that the following holds. Let $I\subset[n]$ satisfy  $r_0:=n-|I|\le d/4$, 
let $A\in \mathcal{M}_{n,d}$ be deterministic, and set $q:=Q\bigl(\operatorname{D}(A)+1\bigr)$. 
If $h\ge2q+h_{\ref{prop:fixed-defect}}$,
then
\[
\mathbf P_\Xi\bigl\{
        \operatorname{corank}(B_I)=h,\,
        \mathcal S_I^c
\bigr\}
\le
        e^{-c_{\ref{prop:fixed-defect}}nh}.
\]
\end{proposition}
\begin{remark}
Proposition~\ref{prop:fixed-defect} is stated for $r_0\le d/4$,
whereas its application below only requires $r_0\le K\log n$. 
Indeed, the row deletion argument is carried out under the condition  $r_0+k\le K\log n$. 
Since $d\ge\lambda n$, this implies $r_0\le d/4$ for all sufficiently large $n$. The wider range in the proposition is convenient because its proof only uses the bound $|C_j|\ge d-r_0\ge3d/4$ for every $j\in[n]$.
\end{remark}
\begin{proof}
Choose $h_{\ref{prop:fixed-defect}}$ sufficiently large in terms of the
fixed parameters, with $h_{\ref{prop:fixed-defect}}\ge Qk_{\ref{lem:hunter-relative}}$,
and set $d_0:=d-r_0$. Since deleting the rows indexed by $I^c$ removes at most $r_0$ support
coordinates from each column, $d_0$ is a uniform lower bound for the
sizes of the restricted column supports.
For every $j\in[n]$,  $|C_j|\ge d_0\ge{3d}/{4}\ge {3\lambda n}/{4}$. 
Conditional on $A$, the columns $X_1,\dots,X_n$ of $B_I$ are independent.

We rewrite the event in the proposition in terms of the column
space of $B_I$. Since $X_1,\dots,X_n$ are the columns of $B_I$, $\rank(B_I)
        =
        \dim\operatorname{span}(X_1,\dots,X_n)$. 
Moreover, for every $w\in\bbR^I$, $B_I^Tw
        =
        \bigl(
        \langle X_1,w\rangle,\dots,\langle X_n,w\rangle
        \bigr)^T$, 
and hence 
        $\ker(B_I^T)
        =
        \operatorname{span}(X_1,\dots,X_n)^\perp$. 

Let $\mathcal V$ denote the family of $(|I|-h)$-dimensional subspaces $V\subset\bbR^I$ such that $V$ is spanned by vectors in $\{0,\pm1\}^I$ and $V^\perp$ contains no nonzero vector supported on at most $\delta n$ coordinates. This family is finite, and every $V\in\mathcal V$ is defined over $\mathbb Q$. Since each column $X_j$ takes values in $\{0,\pm1\}^I$, we have
\begin{equation}\label{eq:fixed-defect-span}
\mathbf P_\Xi\{\operatorname{corank}(B_I)=h,\,\mathcal S_I^c\}=\mathbf P_\Xi\{\operatorname{span}(X_1,\dots,X_n)\in\mathcal V\}.
\end{equation}

For $V\in\mathcal V$ and $j\in[n]$, define $\rho_j(V)
        :=
        \mathbf P_\Xi\{X_j\in V\}$, 
        $\sigma_j(V)
        :=
        \mathbf P\{Y_j\in V\}$. 
Furthermore, define its ``bad'' column set by
\[
        \mathcal B(V)
        :=
        \left\{
        j\in[n]:
        \rho_j(V)>
        \gamma_{\ref{lem:hunter-relative}}^{h/Q}\sigma_j(V)
        \right\}.
\]
Since $h\ge h_{\ref{prop:fixed-defect}}\ge Qk_{\ref{lem:hunter-relative}}$, Proposition~\ref{prop:relative-regular-support} gives $|\mathcal B(V)|<q$. 
For $\mathcal B\subset[n]$ with $|\mathcal B|<q$, define 
        $\mathcal V(\mathcal B)
        :=
        \left\{
        V\in\mathcal V:
        \mathcal B(V)=\mathcal B
        \right\}$. 
These classes partition $\mathcal V$, and their number is at most $n^q$. 

Fix $\mathcal B\subset[n]$ with $|\mathcal B|<q$, and set $J:=[n]\setminus\mathcal B$, $N:=|J|$, $s:=N-(|I|-h)=r_0+h-|\mathcal B|$, 
Since $h\ge2q+h_{\ref{prop:fixed-defect}}$, we have $s\ge h/2$. Thus, on the event $\operatorname{span}(X_1,\dots,X_n)=V$, any basis of $V$ selected from the columns of $B_I$ leaves at least $h/2$ columns indexed by $J$ outside the basis, all of which belong to $V$. By the definition of $\mathcal B(V)$, we have
\begin{equation}\label{eq:relative-good-columns}
        \rho_j(V)
        \le
        \gamma_{\ref{lem:hunter-relative}}^{h/Q}\sigma_j(V),
        \qquad
        j\in J.
\end{equation}

Since the column distributions depend on their supports, the probabilities $\rho_j(V)$ and $\sigma_j(V)$ may vary with $j$. We define their averages over the columns indexed by $J$:
\[
        \overline\rho(V):=\frac1N\sum_{j\in J}\rho_j(V),
        \qquad
        \overline\sigma(V):=\frac1N\sum_{j\in J}\sigma_j(V).
\]
For $0\le t\le N$, we define
\[
        e_t(V)
        :=
        \sum_{\substack{S\subset J\\|S|=t}}
        \prod_{j\in S}\rho_j(V).
\]
In particular, $e_s(V)$ controls the probability that the required $s$ columns indexed by $J$ belong to $V$. When the column distributions are identical, this reduces to $\binom Nt\rho(V)^t$.

Set $\tau:=(1-2p)^{1/4}$. Then, we define
\[
        \mathcal V_{\rm thin}(\mathcal B)
        :=
        \left\{
        V\in\mathcal V(\mathcal B):
        e_s(V)\le\binom Ns\tau^{d_0s}
        \right\},
\qquad
        \mathcal V_{\rm thick}(\mathcal B)
        :=
        \mathcal V(\mathcal B)
        \setminus
        \mathcal V_{\rm thin}(\mathcal B).
\]
We bound the probability that $\operatorname{span}(X_1,\dots,X_n)$ belongs to $V_{\rm thin}$ and to $V_{\rm thick}$, respectively:

\begin{claim}\label{clm:thin}
If $h\ge2q+h_{\ref{prop:fixed-defect}}$, then, for every $\mathcal B\subset[n]$ with $|\mathcal B|<q$,
\[
\mathbf P_\Xi
\left\{
        \operatorname{span}(X_1,\dots,X_n)
        \in\mathcal V_{\rm thin}(\mathcal B)
\right\}\le
e^{-c_{\ref{clm:thin}}nh}.
\]
\end{claim}
\begin{claim}\label{clm:thick}
If $h\ge2q+h_{\ref{prop:fixed-defect}}$, then, for every $\mathcal B\subset[n]$ with $|\mathcal B|<q$, 
\[
\mathbf P_\Xi
\left\{
        \operatorname{span}(X_1,\dots,X_n)
        \in\mathcal V_{\rm thick}(\mathcal B)
\right\}\le
e^{-c_{\ref{clm:thick}}nh}.
\]
\end{claim}
For every $\mathcal B\subset[n]$ with $|\mathcal B|<q$, Claims~\ref{clm:thin} and~\ref{clm:thick} give
\[
\begin{aligned}
&\mathbf P\left\{
         \operatorname{span}(X_1,\dots,X_n)
        \in \mathcal V(B)
\right\}\\
&=\mathbf P_\Xi
\left\{
        \operatorname{span}(X_1,\dots,X_n)
        \in\mathcal V_{\rm thin}(\mathcal B)
\right\}+
\mathbf P_\Xi
\left\{
        \operatorname{span}(X_1,\dots,X_n)
        \in\mathcal V_{\rm thick}(\mathcal B)
\right\}\\
&\le
        e^{-c_{\ref{clm:thin}}nh}
        +
        e^{-c_{\ref{clm:thick}}nh}.
\end{aligned}
\]
The number of subsets $\mathcal B\subset[n]$ satisfying $|\mathcal B|<q$ is at most $n^q$. Hence, using \eqref{eq:fixed-defect-span} and summing over all such $\mathcal B$, we
obtain
\[
\begin{aligned}
&\mathbf P\left\{
        \operatorname{corank}(B_I)=h,\ S_I^c
\right\}\\
&\le
n^q (e^{-c_{\ref{clm:thin}}nh}
        +
        e^{-c_{\ref{clm:thick}}nh})=e^{q\log n}(e^{-c_{\ref{clm:thin}}nh}
        +
        e^{-c_{\ref{clm:thick}}nh})\le e^{-c_{\ref{prop:fixed-defect}}nh}.
\end{aligned}
\]
The last inequality follows from $q\le h/2$. It remains to prove the two claims.

\end{proof}

\begin{proof}[Proof of Claim~\ref{clm:thin}]
If $\operatorname{span}(X_1,\dots,X_n)=V$ for some $V\in\mathcal V_{\rm thin}(\mathcal B)$, then $\dim V=|I|-h$. Hence, there exists $I_0\subset[n]$ with $|I_0|=|I|-h$ such that $\operatorname{span}(X_j:j\in I_0)=V$. For fixed $I_0$ and $V$, we obtain
\begin{align*} 
&\mathbf P_\Xi
\left\{
        \operatorname{span}(X_j:j\in I_0)=V,\
        X_j\in V\text{ for every }j\notin I_0
\right\}\\ &=
\mathbf P_\Xi
\left\{
        \operatorname{span}(X_j:j\in I_0)=V
\right\}
\prod_{j\notin I_0}\rho_j(V).
\end{align*}
The set $J\setminus I_0$ contains at least $N-(|I|-h)$ indices. Since $0\le\rho_j(V)\le1$,
\[
        \prod_{j\notin I_0}\rho_j(V)
        \le
        \prod_{j\in J\setminus I_0}\rho_j(V)
        \le
        e_s(V).
\]
Thus,
\[
\begin{aligned}
&\mathbf P_\Xi
\left\{
        \operatorname{span}(X_1,\dots,X_n)
        \in\mathcal V_{\rm thin}(\mathcal B)
\right\} \\&=
\sum_{V\in\mathcal V_{\rm thin}(\mathcal B)}
\mathbf P_\Xi
\left\{
        \operatorname{span}(X_1,\dots,X_n)=V
\right\} \\
&\le
\sum_{\substack{I_0\subset[n]\\|I_0|=|I|-h}}
\sum_{V\in\mathcal V_{\rm thin}(\mathcal B)}
\mathbf P_\Xi
\left\{
        \operatorname{span}(X_j:j\in I_0)=V
\right\}
e_s(V) \\
&\le
\binom n{|I|-h}\binom Ns\tau^{d_0s}
\le
4^n\tau^{d_0s}.
\end{aligned}
\]
Since $d_0\ge3\lambda n/4$ and $s\ge h/2$, we have
\[
        4^n\tau^{d_0s}
        \le
        \exp\left(
        2n\log2+\frac{3\lambda}{8}nh\log\tau
        \right)
        \le
        e^{-c_{\ref{clm:thin}}nh}.
\]
Since $\tau\in(0,1)$, taking $h_{\ref{prop:fixed-defect}}$ sufficiently large gives the last
inequality for some $c_{\ref{clm:thin}}>0$.
\end{proof}

\begin{proof}[Proof of Claim~\ref{clm:thick}]
By Maclaurin's inequality \cite[Theorem~52]{MR0046395},
\[
        \left(
        \frac{e_s(V)}{\binom Ns}
        \right)^{1/s}
        \le
        \frac1N\sum_{j\in J}\rho_j(V)
        =
        \overline\rho(V).
\]
Hence, every $V\in\mathcal V_{\rm thick}(\mathcal B)$ satisfies
\begin{equation}\label{eq:rho-thick}
        \overline\rho(V)>\tau^{d_0}.
\end{equation}
On the other hand, averaging \eqref{eq:relative-good-columns} over $j\in J$ gives $\overline\rho(V)
        \le
        \gamma_{\ref{lem:hunter-relative}}^{h/Q}
        \overline\sigma(V)$. 

Let $\widetilde Y$ be obtained by first choosing $j$ uniformly from $J$ and then sampling $Y_j$. Then, $\mathbf P\{\widetilde Y\in V\} =\overline\sigma(V)$.
Set 
        $m:=\left\lfloor {d_0}/{2}\right\rfloor$. 
If $W\subset\bbR^I$ has dimension $r\le m$, then 
        $\operatorname{codim}_{E_j}(W\cap E_j)
        \ge
        |C_j|-r
        \ge
        d_0-r$. 
Since the largest atom of $\nu_p$ equals $1-2p$,
Lemma~\ref{lem:odlyzko} gives
\[
        \mathbf P\{Y_j\in W\}
        \le
        (1-2p)^{d_0-r},
        \qquad j\in J.
\]
Averaging over $j\in J$, we obtain
\begin{equation}\label{eq:mixture-odlyzko}
        \mathbf P\{\widetilde Y\in W\}
        \le
        (1-2p)^{d_0-r}.
\end{equation}

Let $\widetilde Y_1,\dots,\widetilde Y_m$ be independent copies of
$\widetilde Y$, independent of $X_1,\dots,X_n$. Fix $V\in\mathcal V_{\rm thick}(\mathcal B)$. Since $\overline\rho(V)
        \le
        \gamma_{\ref{lem:hunter-relative}}^{h/Q}
        \overline\sigma(V)$,
\begin{align}
\mathbf P
\left\{
        \operatorname{span}
        (X_1,\dots,X_n,\widetilde Y_1,\dots,\widetilde Y_m)
        =V
\right\}&\ge
\mathbf P
\left\{
        \operatorname{span}(X_1,\dots,X_n)=V,\,
        \widetilde Y_1,\dots,\widetilde Y_m\in V
\right\} \notag\\
&=
\mathbf P_\Xi
\left\{
        \operatorname{span}(X_1,\dots,X_n)=V
\right\}
\overline\sigma(V)^m \notag\\
&\ge
\gamma_{\ref{lem:hunter-relative}}^{-hm/Q}
\mathbf P_\Xi
\left\{
        \operatorname{span}(X_1,\dots,X_n)=V
\right\}
\overline\rho(V)^m.
\label{eq:thick-lower}
\end{align}

We next bound the probability $\mathbf P
        \left\{
        \operatorname{span}
        (X_1,\dots,X_n,\widetilde Y_1,\dots,\widetilde Y_m)=V
        \right\} $ from above. 
        
If $\operatorname{span}(X_1,\dots,X_n,\widetilde Y_1,\dots,
\widetilde Y_m)=V$, we can first select a basis of the span of the
comparison vectors and then extend it to a basis of $V$ using the
columns $X_1,\dots,X_n$. If the first basis contains $r$ vectors,
then the extension requires $|I|-h-r$ original columns.

For $0\le r\le m$, set 
        $W_r:=\operatorname{span}(\widetilde Y_1,\dots,\widetilde Y_r)$,
         $W_0:=\{0\}$. Since the comparison vectors are independent and identically
distributed, each selection of $r$ comparison indices gives the same
probability. Then, we obtain
\begin{equation}\label{eq:thick-basis-decomposition}
\begin{aligned}
&\mathbf P
\left\{
        \operatorname{span}
        (X_1,\dots,X_n,\widetilde Y_1,\dots,\widetilde Y_m)=V
\right\}\\
&\quad\le
\sum_{r=0}^{\min\{m,|I|-h\}}
\binom mr
\sum_{\substack{I_0\subset[n]\\|I_0|=|I|-h-r}}
\mathbf P
\left\{
\begin{array}{l}
\operatorname{span}
\bigl((X_j)_{j\in I_0},
      \widetilde Y_1,\dots,\widetilde Y_r\bigr)=V,\\[1mm]
X_j\in V
\text{ for all }j\in[n]\setminus I_0,\\[1mm]
\widetilde Y_\ell\in W_r
\text{ for all }r<\ell\le m
\end{array}
\right\}.
\end{aligned}
\end{equation}
Fix $r$ and $I_0$ in this sum, and define $\mathcal F_0
        :=
        \sigma\bigl((X_j)_{j\in I_0},
        \widetilde Y_1,\dots,\widetilde Y_r\bigr)$. The subspace $W_r$ and the event $\left\{
        \operatorname{span}
        \bigl((X_j)_{j\in I_0},
        \widetilde Y_1,\dots,\widetilde Y_r\bigr)=V
        \right\}$ 
are determined by $\mathcal F_0$. Conditional on $\mathcal F_0$,
the vectors $(X_j)_{j\notin I_0}$ and
$\widetilde Y_{r+1},\dots,\widetilde Y_m$ remain mutually independent
with their original distributions.

Since $|J\setminus I_0|\ge N-(|I|-h-r)=s+r$, and $s+r\le N$, independence and Maclaurin's inequality give
\[
\begin{aligned}
\mathbf P
\left\{
        X_j\in V\text{ for all }j\in[n]\setminus I_0
        \,\middle|\,\mathcal F_0
\right\}
&=
        \prod_{j\notin I_0}\rho_j(V)
        \le
        \prod_{j\in J\setminus I_0}\rho_j(V)
        \le
        e_{s+r}(V)\\
&\le
        \binom N{s+r}\overline\rho(V)^{s+r}
        \le
        2^n\overline\rho(V)^{s+r}.
\end{aligned}
\]

Moreover, $\dim W_r\le r$. For every $r<\ell\le m$,
\eqref{eq:mixture-odlyzko} gives
\[
        \mathbf P
        \left\{
        \widetilde Y_\ell\in W_r
        \,\middle|\,\mathcal F_0
        \right\}
        \le(1-2p)^{d_0-\dim{W_r}}\le
        (1-2p)^{d_0-r}.
\]
Since $r\le m\le d_0/2$ and $\tau=(1-2p)^{1/4}$,
\eqref{eq:rho-thick} implies 
        $(1-2p)^{d_0-r}
        \le
        (1-2p)^{d_0/2}
        <
        \tau^{d_0}
        <
        \overline\rho(V)$. 
By independence,
\[
\begin{aligned}
\mathbf P
\left\{
        \widetilde Y_\ell\in W_r
        \text{ for all }r<\ell\le m
        \,\middle|\,\mathcal F_0
\right\}=
        \prod_{\ell=r+1}^m
        \mathbf P
        \left\{
        \widetilde Y_\ell\in W_r
        \,\middle|\,\mathcal F_0
        \right\}
        \le
        \overline\rho(V)^{m-r}.
\end{aligned}
\]
For $r=m$, the event is certain and the product equals one, so the
same bound holds.

Combining these conditional estimates, we obtain
\[
\begin{aligned}
&\mathbf P
\left\{
\operatorname{span}
\bigl((X_j)_{j\in I_0},
      \widetilde Y_1,\dots,\widetilde Y_r\bigr)=V,
X_j\in V
\text{ for all }j\in[n]\setminus I_0,
\widetilde Y_\ell\in W_r
\text{ for all }r<\ell\le m
\right\}\\
&\quad=
\mathbf E\bigg[
\mathbf 1_{\left\{
        \operatorname{span}
        \bigl((X_j)_{j\in I_0},
        \widetilde Y_1,\dots,\widetilde Y_r\bigr)=V
        \right\}}\mathbf P\left\{X_j\in V\text{ for all }j\in[n]\setminus I_0\,\middle|\,\mathcal F_0\right\}
\\
&\qquad\qquad\mathbf P\left\{\widetilde Y_\ell\in W_r\text{ for all }r<\ell\le m\,\middle|\,\mathcal F_0\right\}
\bigg]\\
&\quad\le
        2^n\overline\rho(V)^{s+r}\overline\rho(V)^{m-r}
        \mathbf P
        \left\{
        \operatorname{span}
        \bigl((X_j)_{j\in I_0},
        \widetilde Y_1,\dots,\widetilde Y_r\bigr)=V
        \right\}\\
&\quad=
        2^n\overline\rho(V)^{m+s}
        \mathbf P
        \left\{
        \operatorname{span}
        \bigl((X_j)_{j\in I_0},
        \widetilde Y_1,\dots,\widetilde Y_r\bigr)=V
        \right\}.
\end{aligned}
\]
Combining this bound with \eqref{eq:thick-basis-decomposition} and using
$\overline\rho(V)^s\le1$, we obtain
\[
\begin{aligned}
&\mathbf P
\left\{
        \operatorname{span}
        (X_1,\dots,X_n,\widetilde Y_1,\dots,\widetilde Y_m)=V
\right\}\\
&\quad\le
        2^n\overline\rho(V)^m
        \sum_{r=0}^{\min\{m,|I|-h\}}
        \binom mr
        \sum_{\substack{I_0\subset[n]\\|I_0|=|I|-h-r}}
        \mathbf P
        \left\{
        \operatorname{span}
        \bigl((X_j)_{j\in I_0},
        \widetilde Y_1,\dots,\widetilde Y_r\bigr)=V
        \right\}.
\end{aligned}
\]
Combining this with \eqref{eq:thick-lower}, we obtain
\[
\begin{aligned}
&\mathbf P_\Xi
\left\{
        \operatorname{span}(X_1,\dots,X_n)=V
\right\}\\
&\quad\le
        2^n\gamma_{\ref{lem:hunter-relative}}^{hm/Q}
        \sum_{r=0}^{\min\{m,|I|-h\}}
        \binom mr
        \sum_{\substack{I_0\subset[n]\\|I_0|=|I|-h-r}}
        \mathbf P
        \left\{
        \operatorname{span}
        \bigl((X_j)_{j\in I_0},
        \widetilde Y_1,\dots,\widetilde Y_r\bigr)=V
        \right\}.
\end{aligned}
\]

Summing the preceding bound over
$V\in\mathcal V_{\rm thick}(\mathcal B)$ gives
\[
\begin{aligned}
\mathbf P_\Xi
\left\{
        \operatorname{span}(X_1,\dots,X_n)
        \in\mathcal V_{\rm thick}(\mathcal B)
\right\}
&\le
        2^n\gamma_{\ref{lem:hunter-relative}}^{hm/Q}
        \sum_{r=0}^{\min\{m,|I|-h\}}
        \binom mr\binom n{|I|-h-r}\\
&\le
        2^{m+2n}\gamma_{\ref{lem:hunter-relative}}^{hm/Q}
        \le
        2^{3n}\gamma_{\ref{lem:hunter-relative}}^{hm/Q}.
\end{aligned}
\]
Here we used
\[
        \sum_{r=0}^{\min\{m,|I|-h\}}
        \binom mr\binom n{|I|-h-r}
        \le
        2^n\sum_{r=0}^m\binom mr
        =
        2^{m+n}.
\]

Since $d_0\ge3\lambda n/4$, we have
$m=\lfloor d_0/2\rfloor\ge3\lambda n/16$ for all sufficiently
large $n$. Hence
\[
\begin{aligned}
        2^{3n}\gamma_{\ref{lem:hunter-relative}}^{hm/Q}\le
        \exp\left(
        3n\log2+
        \frac{3\lambda}{16Q}nh
        \log\gamma_{\ref{lem:hunter-relative}}
        \right)\le
        e^{-c_{\ref{clm:thick}}nh}.
\end{aligned}
\]
Taking $h_{\ref{prop:fixed-defect}}$ sufficiently large gives the
last inequality for some $c_{\ref{clm:thick}}>0$ depending only on
the fixed parameters.
\end{proof}

We now combine Proposition~\ref{prop:fixed-defect} with Lemma~\ref{lem:defect-tail} to obtain the corresponding bound for the random support.
\begin{lemma}\label{lem:unstructured-corank}
Fix $K>0$. There exists $c_{\ref{lem:unstructured-corank}}>0$,
depending only on the fixed parameters and $K$, such that the following
holds. Let $I\subset[n]$, set $r_0:=n-|I|$, and suppose that
$r_0\le K\log n$. Then, for every integer $1\le h\le K\log n$,
\[
        \mathbf P_{A,\Xi}
        \left\{
        \operatorname{corank}(B_I)=h,\,
        \mathcal S_I^c
        \right\}
        \le e^{-c_{\ref{lem:unstructured-corank}}nh}
\]
for all sufficiently large $n$.
\end{lemma}

\begin{proof}
Since $d\ge\lambda n$, we have $r_0\le K\log n\le d/4$ for all
sufficiently large $n$. Set $h_1:=2h_{\ref{prop:fixed-defect}}+8Q$. 
If $1\le h<h_1$, every nonzero vector in $\ker(B_I^T)$ extends by zero outside $I$ to a nonzero vector in $\ker(B^T)$. Hence, Theorem~\ref{thm:ssv} gives
\[
\begin{aligned}
        \mathbf P_{A,\Xi} \left\{
        \operatorname{corank}(B_I)=h,\,
        \mathcal S_I^c
        \right\} \le
        \mathbf P_{A,\Xi}\{\rank(B)<n\}\le
        e^{-c_{\ref{thm:ssv}}n}
        \le
        e^{-(c_{\ref{thm:ssv}}/h_1)nh}.
\end{aligned}
\]

Assume that $h\ge h_1$, and set
\[
        r_h
        :=
        \left\lfloor
        \frac{h-h_{\ref{prop:fixed-defect}}}{2Q}
        \right\rfloor-1.
\]
For every support $A$ satisfying $\operatorname{D}(A)\le r_h$, $2Q\left(\operatorname{D}(A)+1\right)
        +h_{\ref{prop:fixed-defect}}
        \le h$. 
Therefore, Proposition~\ref{prop:fixed-defect} gives
\[
        \mathbf P_{A,\Xi}\{\operatorname{corank}(B_I)=h, \mathcal S_I^c\mid A\}
        \le e^{-c_{\ref{prop:fixed-defect}}nh}
\]
on the event $\{\operatorname{D}(A)\le r_h\}$.
Moreover, the choice of $h_1$ implies
\[
        r_h+1
        =
        \left\lfloor
        \frac{h-h_{\ref{prop:fixed-defect}}}{2Q}
        \right\rfloor
        \ge
        \frac{h-h_{\ref{prop:fixed-defect}}}{2Q}-1
        \ge
        \frac{h}{4Q}.
\]
Since $1\le r_h+1\le h\le K\log n$,
Lemma~\ref{lem:defect-tail} gives
\[
        \mathbf P_A\{\operatorname{D}(A)>r_h\}
        \le
        e^{-c_{\ref{lem:defect-tail}}n(r_h+1)}
        \le
        e^{-(c_{\ref{lem:defect-tail}}/(4Q))nh}.
\]

Averaging over $A$, we obtain
\[
\begin{aligned}
        \mathbf P_{A,\Xi}\{ \operatorname{corank}(B_I)=h, \mathcal S_I^c\}
&=
        \mathbf E_A
        \left[
        \mathbf P\{ \left\{         \operatorname{corank}(B_I)=h         \right\}         \cap\mathcal S_I^c\mid A\}
        \right]\\
&\le
        e^{-c_{\ref{prop:fixed-defect}}nh}
        +
        \mathbf P_A\{\operatorname{D}(A)>r_h\}\\
&\le
        e^{-c_{\ref{prop:fixed-defect}}nh}
        +
        e^{-(c_{\ref{lem:defect-tail}}/(4Q))nh}\\
&\le
        e^{-c_{\ref{lem:unstructured-corank}}nh},
\end{aligned}
\]
where the last inequality holds for all sufficiently large $n$ after choosing $c_{\ref{lem:unstructured-corank}}>0$ sufficiently small. Taking this constant also at most $c_{\ref{thm:ssv}}/h_1$ includes the range $1\le h<h_1$ and completes the proof.
\end{proof}

\subsection{Completion of the proof}

We now combine the sparse kernel reduction with the estimate on
$\mathcal S_I^c$ to complete the row deletion argument.

\begin{proposition}\label{prop:row-deletion}
There exist constants
$K_{\ref{prop:row-deletion}},c_{\ref{prop:row-deletion}}>0$, depending
only on $\lambda$, such that the following holds. Let $I\subset[n]$,
set $r_0:=n-|I|$, and let $k\ge1$ be an integer satisfying $r_0+k\le K_{\ref{prop:row-deletion}}\log n$. Then, for all sufficiently large $n$,
\[
        \mathbf P_{A,\Xi}
        \left\{
        \operatorname{corank}(B_I)\ge k
        \right\}
        \le e^{-c_{\ref{prop:row-deletion}}nk}.
\]
\end{proposition}

\begin{proof}
Choose $K_{\ref{prop:row-deletion}}\ge K_{\ref{lem:large-corank}}$
and apply Lemma~\ref{lem:unstructured-corank} with
$K=K_{\ref{prop:row-deletion}}$. Fix 
        $0<c_{\ref{prop:row-deletion}}
        \le \min\left\{
        c_{\ref{thm:ssv}},
        c_{\ref{lem:sparse-kernel-reduction}},
        c_{\ref{lem:unstructured-corank}},
        c_{\ref{lem:large-corank}}
        \right\}/2$, 
and set $\ell:=\lfloor K_{\ref{prop:row-deletion}}\log n\rfloor$.
We assume that $n$ is sufficiently large; in particular, $\ell<n$.

Extending vectors by zero outside $I$ gives an injective linear map
from $\ker(B_I^T)$ into $\ker(B^T)$. Hence
\begin{equation}\label{eq:row-corank-comparison}
        \operatorname{corank}(B_I)
        =\dim\ker(B_I^T)
        \le\dim\ker(B^T)
        =\operatorname{corank}(B).
\end{equation}

We argue by induction on $k$, proving the assertion for every $I$ in
the stated range. If $k>|I|$, the event is empty, so we may assume
$k\le |I|$.

If $k=1$, \eqref{eq:row-corank-comparison} and
Theorem~\ref{thm:ssv} give
\[
        \mathbf P_{A,\Xi}\{\mathcal E_1(I)\}
        \le
        \mathbf P_{A,\Xi}\{\rank(B)<n\}
        \le
        e^{-c_{\ref{thm:ssv}}n}
        \le
        e^{-c_{\ref{prop:row-deletion}}n}.
\]

Assume that $k\ge2$ and that the assertion holds for $k-1$. For every
$i\in I$, we have $n-|I\setminus\{i\}|=r_0+1$ and
        $(r_0+1)+(k-1)
        =
        r_0+k
        \le
        K_{\ref{prop:row-deletion}}\log n$.
Thus, the induction hypothesis applies to every $B_{I\setminus\{i\}}$.

Since $k\le\ell$, we have
\begin{equation}\label{eq:row-deletion-decomposition}
\begin{aligned}
\mathbf P_{A,\Xi}\{\mathcal E_k(I)\}
&= \mathbf P_{A,\Xi}\{\mathcal E_k(I)\cap\mathcal S_I\} +
\sum_{h=k}^{\ell}
\mathbf P_{A,\Xi}
\left\{
        \operatorname{corank}(B_I)=h,\,
        \mathcal S_I^c
\right\} \\&\qquad+
\mathbf P_{A,\Xi}
\left\{
        \operatorname{corank}(B_I)>\ell,\,
        \mathcal S_I^c
\right\}.
\end{aligned}
\end{equation}

For the first term, Lemma~\ref{lem:sparse-kernel-reduction} and the
induction hypothesis give
\[
\begin{aligned}
\mathbf P_{A,\Xi}\{\mathcal E_k(I)\cap\mathcal S_I\}
\le
e^{-c_{\ref{lem:sparse-kernel-reduction}}n}
\sup_{i\in I}
\mathbf P_{A,\Xi}
\left\{
        \mathcal E_{k-1}(I\setminus\{i\})
\right\}\le
e^{-c_{\ref{lem:sparse-kernel-reduction}}n}
e^{-c_{\ref{prop:row-deletion}}n(k-1)}
\le
\frac12e^{-c_{\ref{prop:row-deletion}}nk}.
\end{aligned}
\]
The last inequality holds for all sufficiently large $n$, since
$c_{\ref{prop:row-deletion}}\le
c_{\ref{lem:sparse-kernel-reduction}}/2$.

For the second term, Lemma~\ref{lem:unstructured-corank} gives
\[
\begin{aligned}
\sum_{h=k}^{\ell}
\mathbf P_{A,\Xi}
\left\{
        \operatorname{corank}(B_I)=h,\,
        \mathcal S_I^c
\right\}\le
\sum_{h=k}^{\ell}
e^{-c_{\ref{lem:unstructured-corank}}nh}
\le
\frac{e^{-c_{\ref{lem:unstructured-corank}}nk}}
     {1-e^{-c_{\ref{lem:unstructured-corank}}n}}
\le
\frac14e^{-c_{\ref{prop:row-deletion}}nk}.
\end{aligned}
\]
Here we used $r_0\le K_{\ref{prop:row-deletion}}\log n$ and
$h\le\ell\le K_{\ref{prop:row-deletion}}\log n$. The last inequality
holds for all sufficiently large $n$ by the choice of
$c_{\ref{prop:row-deletion}}$.

For the third term, \eqref{eq:row-corank-comparison} implies
\[
\begin{aligned}
\mathbf P_{A,\Xi}
\left\{
        \operatorname{corank}(B_I)>\ell,\,
        \mathcal S_I^c
\right\}\le
\mathbf P_{A,\Xi}
\left\{
        \operatorname{corank}(B)\ge\ell+1
\right\}
\le
e^{-c_{\ref{lem:large-corank}}n(\ell+1)}
\le
\frac14e^{-c_{\ref{prop:row-deletion}}nk}.
\end{aligned}
\]
The second inequality follows from Proposition~\ref{lem:large-corank}, since $\ell+1>K_{\ref{prop:row-deletion}}\log n
        \ge K_{\ref{lem:large-corank}}\log n$.
The last inequality holds for all sufficiently large $n$, using
$k\le\ell$ and
$c_{\ref{prop:row-deletion}}\le c_{\ref{lem:large-corank}}/2$.

Combining these bounds with
\eqref{eq:row-deletion-decomposition} gives
\[
        \mathbf P_{A,\Xi}\left\{\operatorname{corank}(B_I)\ge k\right\}=\mathbf P_{A,\Xi}\{\mathcal E_k(I)\}
        \le e^{-c_{\ref{prop:row-deletion}}nk}.
\]
This completes the induction.
\end{proof}

\begin{proof}[Proof of Theorem~\ref{thm:rank}]
For $k\ge K_{\ref{prop:row-deletion}}\log n$, the required estimate
follows from Proposition~\ref{lem:large-corank}. For
$1\le k<K_{\ref{prop:row-deletion}}\log n$, apply
Proposition~\ref{prop:row-deletion} with $I=[n]$. Choosing 
        $0<c_{\ref{thm:rank}}
        \le
        \min\left\{
        c_{\ref{lem:large-corank}},
        c_{\ref{prop:row-deletion}}
        \right\}$, 
we obtain
\[
        \mathbf P_{A,\Xi}
        \left\{
        \rank(A\circ\Xi)\le n-k
        \right\}
        \le e^{-c_{\ref{thm:rank}}nk}
\]
for every $1\le k\le n$ and all sufficiently large $n$.
\end{proof}

\section{Proof of Theorem~\ref{thm:ssv}}\label{sec:ssv} 
We first specify the parameters used in both proofs. Choose $\delta,\rho>0$ sufficiently small in terms of $\lambda$. Set 
$b_1:= {\rho}/{\sqrt2}$, $b_2:=1/{\sqrt\delta}$,  $b_3:={\rho^2\delta}/{2}$. 
Choose $0<\eta<\min\{\lambda^2/16,b_3/4,\delta/2\}$ and take $Q=Q(\lambda,\eta)$ as in Proposition~\ref{prop:support}, sufficiently large that $\lambda Q\ge8\log(e/\eta)$. The remaining parameters are chosen below.

Throughout this section, we fix $\kappa\ge0$, and let $\chi>0$ be chosen at the end of the argument. We may assume that $\kappa\ge e^{-\chi n}$. Indeed, once the desired bound has been established for
$\kappa=e^{-\chi n}$, the event $\{s_n(B)\le\kappa\}$ implies that, for every $0\le\kappa<e^{-\chi n}$,
\[
        \mathbf P_{A,\Xi}\{s_n(A\circ\Xi)\le\kappa\}
        \le
        \mathbf P_{A,\Xi}\{s_n(A\circ\Xi)\le e^{-\chi n}\}.
\]
Let $\mathcal G$ denote the event that both $\mathcal T(A)$ and $\mathcal T^*(A)$ are $(Q,\eta,\lambda)$-well-spread. Since $\mathbf P_A(\mathcal G^c)\le e^{-c_{\ref{prop:support}}n}$, it suffices to prove that, for every deterministic $A\in\mathcal G$,
\[
        \mathbf P_\Xi\{s_n(A\circ\Xi)\le \kappa\}
        \le C_\lambda\kappa\sqrt n+e^{-c_\lambda n}. 
\]
Indeed,
\[
\mathbf P_{A,\Xi}\{s_n(A\circ\Xi)\le\kappa\}
\le \mathbf P_A(\mathcal G^c)+
\sup_{A\in\mathcal G}\mathbf P_\Xi\{s_n(A\circ\Xi)\le\kappa\}.
\]
From now on, $A\in\mathcal G$ is fixed and all probabilities are with respect to the signs $\Xi$ only.

\subsection{Fixed vectors and compressible directions}

\begin{lemma}\label{lem:fixed-vector}
There exists $c_{\ref{lem:fixed-vector}}>0$, depending only on $\lambda$,
such that, for every fixed $A\in \mathcal{M}_{n,d}$ and every deterministic
$x\in\mathbb S^{n-1}$,
\[
        \mathbf P_\Xi\left\{\|(A\circ\Xi)x\|_2
        \le c_{\ref{lem:fixed-vector}}\sqrt n\right\}
        \le e^{-c_{\ref{lem:fixed-vector}}n}.
\]
Moreover, the same result holds with $A\circ\Xi$ replaced by $(A\circ\Xi)^T$.
\end{lemma}

\begin{proof}
Set $B:=A\circ\Xi$ and
\[
        X_i:=B_i\cdot x=\sum_{j\in T_i(A)}\xi_{ij}x_j,
        \qquad
        \sigma_i^2:=\mathbf E_\Xi X_i^2=\|x|_{T_i(A)}\|_2^2.
\]
Since every column of $A$ contains exactly $d$ ones,
\[
        \sum_{i=1}^n\sigma_i^2
        =\sum_{j=1}^n d x_j^2=d\ge\lambda n,
        \qquad 0\le\sigma_i^2\le1.
\]
The Khintchine inequality for Rademacher sums gives $\mathbf E_\Xi X_i^4\le3\left(\mathbf E_\Xi X_i^2\right)^2=3\sigma_i^4$. 
Then, the Paley--Zygmund inequality in \cite[Lemma 3.5]{MR2146352} states
\[
        \mathbf P_\Xi\{X_i^2>\sigma_i^2/2\}
        \ge\frac{(\mathbf E_\Xi X_i^2)^2}{4\mathbf E_\Xi X_i^4}
        \ge\frac1{12}
\]
whenever $\sigma_i>0$. The case $\sigma_i=0$ is irrelevant in the weighted sum below. Let $\zeta_i:=\mathbf 1_{\{X_i^2> \sigma_i^2/2\}}$.
Since the rows of $B$ are independent after conditioning on $A$, the random variables $\zeta_1,\dots,\zeta_n$ are independent, and $\mathbf E_\Xi \zeta_i\ge {1}/{12}$ for every $i$ with $\sigma_i>0$.

Consider the weighted sum $S:=\sum_{i=1}^n \sigma_i^2\zeta_i$.
Then
\[
        \mathbf E_\Xi S
        =
        \sum_{i=1}^n \sigma_i^2\mathbf E_\Xi\zeta_i
        \ge
        \frac{1}{12}\sum_{i=1}^n\sigma_i^2
        =
        \frac{1}{12}d
        \ge
        \frac{1}{12}\lambda n.
\]
We now recall a concentration inequality for sums of independent, mean-zero, bounded random variables (see, e.g., \cite[Theorem 2.9.5]{MR3837109}).

Let $Y_1,\dots,Y_n$ be independent, mean-zero random variables satisfying $|Y_i| \le K$ all $i$. Then, for every $t \ge 0$, we have
\[
\mathbf{P}\left\{ \left| \sum_{i=1}^n Y_i \right| \ge t \right\}
\le 2\exp\left( -\frac{t^2/2}{\sigma^2 + Kt/3} \right),
\]
where $\sigma^2 = \sum_{i=1}^n \mathbf{E} Y_i^2$ is the variance of the sum.

Apply this to $Y_i:=\sigma_i^2(\zeta_i-\mathbf E_\Xi\zeta_i)$. Since $0\le\sigma_i^2\le 1$, we have $|Y_i|\le1$, and $\sum_{i=1}^n \mathbf E_\Xi Y_i^2\le \sum_{i=1}^n
\sigma_i^4\le\sum_{i=1}^n \sigma_i^2=d\le n$.
Taking $t= d/24$, Bernstein's inequality gives
\[
\mathbf P_\Xi\left\{S\le \frac{d}{24}\right\}\le\mathbf P_\Xi\left\{\sum_{i=1}^n Y_i\le -\frac{d}{24}\right\}\le2\exp(-c n),
\]
where $c>0$ depends only on $\lambda$. Since $\|Bx\|_2^2 =\sum_{i=1}^n X_i^2 \ge \sum_{i=1}^n  {\sigma_i^2} \zeta_i/2 = S/2$, we have 
\[
        \mathbf P_\Xi\left\{\|Bx\|_2\le \sqrt{\frac{d}{48}}\right\}\le2\exp(-c n).
\]
The proof for $B^T$ is identical, using the column supports of $A$, equivalently the row supports of $A^T$.
\end{proof}

We use the standard decomposition of the unit sphere into compressible and incompressible vectors.

\begin{definition}[Compressible and incompressible vectors]\label{def:compressible}
Fix $\delta, \rho \in (0, 1)$. A vector $x \in \mathbb{R}^n$ is called \textit{sparse} if $|\operatorname{supp}(x)| \leq \delta n$. A vector $x \in \mathbb{S}^{n-1}$ is called \textit{compressible} if $x$ is within Euclidean distance $\rho$ from the set of all sparse vectors. A vector $x \in \mathbb{S}^{n-1}$ is called \textit{incompressible} if it is not compressible. The sets of sparse, compressible, and incompressible vectors will be denoted by $\operatorname{Sparse}, \operatorname{Comp}$ and $\operatorname{Incomp}$ respectively.
\end{definition}

Then, we show that incompressible vectors have many coordinates of the order $n^{-1/2}$.

\begin{lemma}[{\cite[Lemma 3.4]{MR2407948}}]\label{lem:spread}
Let $x\in\operatorname{Incomp}(\delta,\rho)$. Then there exists a set $\sigma\subset[n]$ of cardinality $|\sigma| \geq \rho^2\delta n / 2$ and such that
\begin{align*}
\frac{\rho}{\sqrt {2n}}\le |x_k|\le \frac{1}{\sqrt {\delta n}} \qquad \text{for all } k \in \sigma.
\end{align*}
\end{lemma}
By the above lemma, for every $x\in\operatorname{Incomp}(\delta,\rho)$ the set
\[
        \operatorname{Spread}(x):=
        \left\{k\in[n]:
        \frac{\rho}{\sqrt {2n}}\le |x_k|\le \frac{1}{\sqrt {\delta n}}
        \right\}
\]
satisfies $|\operatorname{Spread}(x)|\ge{\rho^2\delta}n/2$.

\begin{proof}
This is \cite[Lemma~3.4]{MR2407948}.
\end{proof}

For $S\subset\bbR^n$ and $t>0$, denote by $N(S,tB_2^n)$ the smallest
number of Euclidean balls of radius $t$ needed to cover $S$. We use the
following net construction due to Livshyts~\cite[Theorem~4]{MR4361906}, in the form stated in\cite[Theorem~3.8]{MR4523249}.

\begin{lemma}[{\cite[Theorem 4]{MR4361906}}]\label{lem:rr}
There exists an absolute constant $C_{\ref{lem:rr}}>0$ such that the
following holds. Let $S\subset\mathbb S^{n-1}$, let $\alpha\in(0,1/2)$ and $\beta\in(0,\alpha/10)$, and suppose that $n\ge1/\alpha^2$. There exists a deterministic net
$\mathcal N\subset S+(4\beta/\alpha)B_2^n$ satisfying
\[
        |\mathcal N|\le N(S,\beta B_2^n)
        \exp\left(C_{\ref{lem:rr}}\alpha^{0.08}
        \log(1/\alpha)n\right).
\]
Moreover, for every deterministic $n\times n$ matrix $M$ and every
$x\in S$, there exists $y\in\mathcal N$ such that
\[
        \|M(x-y)\|_2\le\frac{2\beta}{\alpha\sqrt n}\|M\|_{\HS}.
\]
\end{lemma}

The net in Lemma~\ref{lem:rr} is chosen independently of $M$. This allows us to apply the bound to each realization of $B$ and then take a union bound over a deterministic set.

\begin{lemma}\label{lem:compressible}
For the parameters $\delta,\rho$ fixed above, there exists $c_{\ref{lem:compressible}}>0$, depending only on $\lambda$, such that, for every fixed $A\in \mathcal{M}_{n,d}$,
\[
        \mathbf P_\Xi\left\{\inf_{x\in\Comp(\delta,\rho)}
        \|(A\circ\Xi)x\|_2\le c_{\ref{lem:compressible}}\sqrt n\right\}
        \le e^{-c_{\ref{lem:compressible}}n}.
\]
The same bound holds with $A\circ\Xi$ replaced by $(A\circ\Xi)^T$.
\end{lemma}

\begin{proof}
We prove the first bound. The bound for $(A\circ\Xi)^T$ follows by applying the same argument to the column supports of $A$.

Set $B:=A\circ\Xi$ and $c=\min\{c_{\ref{lem:fixed-vector}}/4,1/10\}$. Set $\alpha={2(\delta+2\rho)}/{c}$ and $\beta = \delta+2\rho$, where $0<\delta,\rho\ll1$. Assume that $n\ge1/\alpha^2$. Observe that
\begin{align*}
N\left(\operatorname{Comp}(\delta,\rho),\beta B_2^n\right)\le N\left(\operatorname{Sparse}(\delta)\cap\mathbb{S}^{n-1},(\beta-2\rho)B_2^n\right)\le\binom{n}{\delta n}\left(\frac{4}{\delta}\right)^{\delta n}
\le e^{2\delta\log(4/\delta)\,n}.
\end{align*}
Using Lemma \ref{lem:rr}, we get a net $\mathcal{N}\subseteq\frac32B_2^n\backslash\frac12B_2^n$ with
\begin{align*}
|\mathcal{N}| \leq e^{2\delta\log(4/\delta)\,n} \cdot e^{C_{\ref{lem:rr}} \alpha^{0.08} \log(1/\alpha) n}\le e^{C \alpha^{0.08} \log(1/\alpha) n}.
\end{align*}
If $\|Bx\|_2\le c\sqrt n$ for some $x\in\Comp(\delta,\rho)$,
Lemma~\ref{lem:rr} gives $y\in\mathcal N$ such that
\[
        \|B(x-y)\|_2\le\frac{2\beta}{\alpha\sqrt n}\|B\|_{\HS}
        \le c\sqrt n,
\]
where we used $\|B\|_{\HS}=\sqrt{nd}\le n$. Thus, $\|By\|_2\le2c\sqrt n$. For each fixed $y\in\mathcal N$, since $\|y\|_2\ge1/2$ and $4c\le c_{\ref{lem:fixed-vector}}$, Lemma~\ref{lem:fixed-vector} implies that
\[
        \mathbf P_\Xi\{\|By\|_2\le2c \sqrt n\}
        \le\mathbf P_\Xi\{\|B(y/\|y\|_2)\|_2
        \le c_{\ref{lem:fixed-vector}}\sqrt n\}
        \le e^{-c_{\ref{lem:fixed-vector}}n}.
\]
Taking a union bound over $\mathcal N$ proves the first estimate after choosing $\alpha$ sufficiently small. 
\end{proof}

We now recall the least common denominator (LCD), which provides the arithmetic input for the Littlewood--Offord estimates used below. Once the support $A$ has been fixed, the randomness in the $i$-th row is the Rademacher sum $(A\circ\Xi)_i\cdot x$. Small ball probabilities for such sums depend on the additive structure of the coefficient vector $x|_{T_i(A)}$.

The restrictions $x|_{T_i(A)}$ do not need to lie on $\mathbb S^{n-1}$. We shall therefore use the standard LCD estimates of Rudelson and Vershynin \cite{MR2407948} in a rescaled form: whenever $v\in\mathbb R^m$ satisfies $\|v\|_2\ge a$, the corresponding estimates are applied to $v/\|v\|_2$ and then transferred back to $v$. The only effect is that the implicit constants may depend on the fixed parameter $a$.

For $v\in\bbR^m$, $\alpha>0$, and $\gamma\in(0,1)$, define
\[
        \LCD_{\alpha,
\gamma}(v)
        :=\inf\left\{\theta>0:
        \dist(\theta v,\bbZ^m)<\min(\gamma\theta\|v\|_2,\alpha)
        \right\}.
\]

We use the following standard Littlewood--Offord estimate in terms of LCD.
\subsection{LCD estimates and arithmetic nets}

\begin{lemma}\label{lem:lcd-smallball}
For every $a>0$ and $\gamma\in(0,1)$ there are constants $C,c>0$ such that the following holds. Let $\xi_1,\dots,\xi_m$ be independent Rademacher random variables. If $v\in\bbR^m$ satisfies $\|v\|_2\ge a$, then for every $\alpha>0$ and every $\eps\ge0$,
\[
        \mathcal L\left(\sum_{j=1}^m\xi_jv_j,\eps\right)
        :=\sup_{u\in\bbR}
        \mathbf P\left\{\left|\sum_{j=1}^m\xi_jv_j-u\right|\le\eps\right\}
        \le C\eps+\frac{C}{\LCD_{\alpha,
\gamma}(v)}+Ce^{-c\alpha^2}.
\]
\end{lemma}
\begin{proof}
This follows from \cite[Theorem 4.1]{MR2407948}.
\end{proof}
The following stability estimate is a direct consequence of the definition of LCD. It is the standard LCD analogue of the CLCD stability lemma of Tran \cite[Lemma 2.14]{tran2020smallestsingularvaluerandom}; see also Jain, Sah, and Sawhney \cite[Lemma 2.12]{MR4523249} for the version used in the random regular digraph setting.

\begin{lemma}\label{lem:lcd-stability}
Let $v,w\in\mathbb R^N$, let $\alpha>0$, and let $\gamma\in(0,1)$. Suppose that $\|v-w\|_2< {\gamma}\|v\|_2/{5}$. 
Then
\[
        \LCD_{\alpha/2,\gamma/2}(w)
        \ge
        \min\left\{
        \LCD_{\alpha,\gamma}(v),
        \frac{\alpha}{4\|v-w\|_2}
        \right\}.
\]
\end{lemma}
\begin{proof}
Let $r:=\|v-w\|_2$. There is nothing to prove if $r=0$, so assume $r>0$. Suppose that
\[
        \LCD_{\alpha/2,\gamma/2}(w)
        <
        \min\left\{\LCD_{\alpha,\gamma}(v),\frac{\alpha}{4r}\right\}.
\]
Then there exists $0<\theta<\min\left\{\LCD_{\alpha,\gamma}(v), {\alpha}/{4r}\right\}$ such that $\dist(\theta w,\mathbb Z^N)<\min\left\{ \frac{\gamma}{2}\theta\|w\|_2, \frac{\alpha}{2}\right\}$.  Choose $p\in\mathbb Z^N$ satisfying $\|\theta w-p\|_2 < \min\left\{ {\gamma}\theta\|w\|_2/2, {\alpha}/{2}\right\}$.
Then
\[
        \|\theta v-p\|_2
        \le
        \|\theta w-p\|_2+\theta\|v-w\|_2
        <
        \min\left\{\frac{\gamma}{2}\theta\|w\|_2,\frac{\alpha}{2}\right\}
        +\theta r.
\]

Since $ {\alpha}/{2}+\theta r <  {\alpha}/{2}+ {\alpha}/{4} =  {3\alpha}/{4} < \alpha$
and $ {\gamma} \theta\|w\|_2/2+\theta r\le {\gamma}\theta\|v\|_2/2
+\left(1+ {\gamma}/{2}\right)\theta r \le\gamma\theta\|v\|_2$, we have
 $\dist(\theta v,\mathbb Z^N)
        \le
        \|\theta v-p\|_2
        <
        \min\{\gamma\theta\|v\|_2,\alpha\}$. 
Hence $\LCD_{\alpha,\gamma}(v)\le \theta<\LCD_{\alpha,\gamma}(v)$, which contradicts the definition of $\LCD_{\alpha,\gamma}(v)$.
\end{proof}

The argument also requires a refined covering estimate for vectors with bounded LCD, playing the role of a volumetric net bound for these arithmetic level sets. 
\begin{lemma}\label{lem:lcd-net}
Fix $a,\theta,\gamma,c_0>0$. There exists a constant
$C_{\ref{lem:lcd-net}}>0$, depending only on
$a,\theta,\gamma,c_0$, such that the following holds. Let $\theta n\le m\le n$,
let $D\ge c_0\sqrt n$, and let $0<\mu<1$. Set $\alpha:=\mu\sqrt n$ and define
\[
        \mathcal L_{m,D}:=
        \left\{
        v\in B_2^m:
        a\le\|v\|_2\le 1,\ 
        D\le\LCD_{\mu\sqrt n,\gamma}(v)\le 2D
        \right\}.
\]
Then $\mathcal L_{m,D}$ admits a Euclidean $({4\mu\sqrt n}/{D})$-net of cardinality at most $\mu^{-1}D\left({C_{\ref{lem:lcd-net}}D}/{\sqrt n}\right)^m$.
\end{lemma}
\begin{remark}
The version in \cite[Lemma 5.6]{MR2407948} is stated only for $\|v\|_2 = 1$, but the statement above is an easy consequence. The key point in the lemma is that the constant $C_{\ref{lem:lcd-net}}$ is independent of $\mu$.   
\end{remark}
The preceding level set estimate also yields a sublevel version whenever a uniform lower bound of order $\sqrt n$ is available.

\begin{corollary}\label{cor:lcd-sublevel-net}
Fix $a,\theta,\gamma,c_0>0$. There exists a constant
$C_{\ref{cor:lcd-sublevel-net}}>0$, depending only on
$a,\theta,\gamma,c_0$, such that the following holds. Let $\theta n\le m\le n$, $D\ge c_0\sqrt n$, $0<\mu<1$, and set $\alpha:=\mu\sqrt n$. Define
\[
        \mathcal S_{m, D}:=
        \left\{
        v\in B_2^m:
        a\le\|v\|_2\le1,\ 
        c_0\sqrt n\le
        \LCD_{\alpha,\gamma}(v)
        \le2D
        \right\}.
\]
Then $\mathcal S_{m, D}$ admits a Euclidean $({4\mu\sqrt n}/{D})$-net of cardinality at most 
$\mu^{-1}D\left( {C_{\ref{cor:lcd-sublevel-net}}D}/{\sqrt n} \right)^m$.
\end{corollary}

\begin{proof}
Let $H$ range over the dyadic values $c_0\sqrt n,\ 2c_0\sqrt n,\ 4c_0\sqrt n,\dots$ up to $2D$. For each such $H$, Lemma~\ref{lem:lcd-net} gives a $4\alpha/H$-net for the level sets
\[
        \mathcal L_{m, H}:=\left\{
        v\in B_2^m:
        a\le\|v\|_2\le 1,
        H\le\LCD_{\alpha,\gamma}(v)\le2H
        \right\}
\]
of cardinality at most $\mu^{-1}H \left({C_{\ref{lem:lcd-net}}H}/{\sqrt n}\right)^m$. 

If $H<D$, each ball of radius $4\alpha/H$ admits a $4\alpha/D$-net of cardinality at most $\left({3D}/{H}\right)^m$. Thus, $\mathcal L_{m, H}$ admits a $4\alpha/D$-net of cardinality at most $\mu^{-1}H\left({CD}/{\sqrt n}\right)^m$. 
For $H\ge D$, the original net already has radius at most $4\alpha/D$, and
the same bound holds after changing the constant in the base.

Summing over the dyadic values of $H$ and using $\sum_{H\le 2D}H\le 4D$, we obtain a Euclidean $4\alpha/D$-net for $\mathcal S_{m, D}$ of cardinality at most $\mu^{-1}D\left( {C_{\ref{cor:lcd-sublevel-net}}D}/{\sqrt n} \right)^m$.
\end{proof}

\subsection{QLCD and nets for incompressible vectors}

Recall the parameters $b_1$, $b_2$, $b_3$, $\eta$, $Q$ and the event $\mathcal G$ fixed at the beginning of this section. For the arguments below, fix a support
$A\in\mathcal G$. Set $B=A\circ\Xi$, set $T_i:=T_i(A)$ and
$\mathcal T=\{\{T_1,\dots,T_n\}\}$. 
The QLCD introduced below will be defined relative to this deterministic family of supports. This is
the ordinary LCD version of the quantile construction in
\cite[Definition~3.3 and the following remark]{MR4523249}.

\begin{definition}[QLCD]\label{def:qlcd}
Let $x\in\mathbb{R}^n$. Given $n$ sets of coordinates $\mathcal T=\{\{T_1,\dots,T_n\}\}$ and $\ell\in[n]$, we define the quantile $\LCD$
or $\QLCD^{\mathcal T}_{\ell,\alpha,\gamma}(x)$ to be the $\ell$-th smallest element in the multiset
\[
        \{\{\LCD_{\alpha,\gamma}(x|_{T_i}):i=1,\dots,n\}\}.
\]
\end{definition}
Throughout the proof, we take $\ell=Q$, $\alpha=\mu\sqrt n$.
The following lemma shows that incompressible vectors have sufficiently large QLCD on a well-spread support family.

\begin{lemma}\label{lem:qlcd-lower}
There exist constants $a_{\ref{lem:qlcd-lower}}, c_{\ref{lem:qlcd-lower}}>0$,
depending only on $\lambda,\delta,\rho$, such that the
following holds. If $\gamma>0$ is sufficiently small, then for every
$\mu>0$ and every $x\in\Incomp(\delta,\rho)$, all but fewer than $Q$
sets $T_i\in\mathcal T$ satisfy
\[
        \|x|_{T_i}\|_2
        \ge
        a_{\ref{lem:qlcd-lower}}.
\]
Moreover,
\[
        \QLCD^{\mathcal T}_{Q,\mu\sqrt n,\gamma}(x)
        \ge
        c_{\ref{lem:qlcd-lower}}\sqrt n .
\]
\end{lemma}

\begin{proof}
Set $S:=\operatorname{Spread}(x)$ and choose $r>0$ with
$Qr<b_3/2$. By Lemma~\ref{lem:spread}, $|S|\ge b_3n$ and
$b_1/\sqrt n\le|x_j|\le b_2/\sqrt n$ for every $j\in S$.
If $Q$ distinct indices $i_1,\dots,i_Q$ satisfied
$|T_{i_k}\cap S|<rn$, then
\[
        \left|[n]\setminus\bigcup_{k=1}^Q T_{i_k}\right|
        \ge\left|S\setminus\bigcup_{k=1}^Q T_{i_k}\right|
        \ge b_3n-Qrn>\frac{b_3n}{2}>\eta n,
\]
contrary to the covering condition in
Definition~\ref{def:well-spread}. Therefore, for all but fewer than $Q$ indices $i$, $|T_i\cap S|\ge rn$.

For every remaining index, $\|x|_{T_i}\|_2^2\ge|T_i\cap S| {b_1^2}/{n}
        \ge rb_1^2$. 
Set $a_{\ref{lem:qlcd-lower}}:=b_1\sqrt r$, choose
$c_{\ref{lem:qlcd-lower}}>0$ such that $c_{\ref{lem:qlcd-lower}}b_2<1/2$, and set $0<\gamma<a_{\ref{lem:qlcd-lower}}$.
For $0<t\le c_{\ref{lem:qlcd-lower}}\sqrt n$, we have
$|tx_j|<1/2$ for every $j\in T_i\cap S$. The closest integer on each of these
coordinates is zero. Hence
\[
        \dist(tx|_{T_i},\bbZ^{T_i})
        \ge t\|x|_{T_i\cap S}\|_2
        \ge tb_1\sqrt r
        >\gamma t\|x|_{T_i}\|_2,
\]
where we used $\|x|_{T_i}\|_2\le1$. Thus, no $t\le c_{\ref{lem:qlcd-lower}}\sqrt n$ can satisfy the LCD condition, and hence
\[
        \LCD_{\mu\sqrt n,\gamma}(x|_{T_i})
        \ge
        c_{\ref{lem:qlcd-lower}}\sqrt n
\]
for all but fewer than $Q$ indices $i$.
\end{proof}

Fix $\gamma>0$ as in Lemma~\ref{lem:qlcd-lower}. Then, we will construct sufficiently small nets for level sets of the QLCD.
For $D\ge c_{\ref{lem:qlcd-lower}}\sqrt n$, define
\[
        K_D:=\{x\in\Incomp(\delta,\rho):
        D\le\QLCD^{\mathcal T}_{Q,\mu\sqrt n,\gamma}(x)\le2D\}.
\]
The next lemma provides the net used in the treatment of the small QLCD regime.

\begin{lemma}\label{lem:qlcd-net}
There exists $C_{\ref{lem:qlcd-net}}>0$, depending only on $\lambda,\delta,\rho,\gamma$, such that the following holds. For every
$D\ge c_{\ref{lem:qlcd-lower}}\sqrt n$ and every sufficiently small
$\mu>0$, there exists a finite set $\mathcal N_D$ satisfying
\[
        \mathcal N_D\subset K_D+
        C_{\ref{lem:qlcd-net}}\frac{\mu\sqrt n}{D}B_2^n,
        \qquad
        |\mathcal N_D|\le n\mu^{-1}D
        \left(\frac{C_{\ref{lem:qlcd-net}}D\mu^{\lambda-1}}{\sqrt n}\right)^n.
\]
Moreover, for every $x\in K_D$ and every deterministic $n\times n$
matrix $M$, there exists $y\in\mathcal N_D$ such that
\[
        \|M(x-y)\|_2\le C_{\ref{lem:qlcd-net}}
        \frac{\mu}{D}\|M\|_{\HS}.
\]
\end{lemma}

\begin{proof}
Let $x\in K_D$. Since $\QLCD^{\mathcal T}_{Q,\mu\sqrt n,\gamma}(x)\le2D$, 
there are at least $Q$ row supports $T_i$ such that $\LCD_{\mu\sqrt n,\gamma}(x|_{T_i})\le2D$. On the other hand, Lemma~\ref{lem:qlcd-lower} shows that fewer than $Q$ row supports fail the corresponding mass and LCD lower bounds. It follows that
there exists a row support $T=T_i$ such that, with $m:=|T| =d\ge\lambda n$, $\|x|_T\|_2\ge a_{\ref{lem:qlcd-lower}}$, and $c_{\ref{lem:qlcd-lower}}\sqrt n \le \LCD_{\mu\sqrt n,\gamma}(x|_T)\le2D$. 

Fix one of such supports $T$. Let $C_{\ref{lem:qlcd-net}}\ge C_{\ref{cor:lcd-sublevel-net}}>0$ be sufficiently large, depending only on
$C_{\ref{cor:lcd-sublevel-net}}$. By Corollary \ref{cor:lcd-sublevel-net}, we obtain a $(4\mu\sqrt n/D)$-net for the possible restrictions $x|_T$ of cardinality at most $\mu^{-1}D\left({C_{\ref{cor:lcd-sublevel-net}}D}/{\sqrt n}\right)^m\le\mu^{-1}D\left({C_{\ref{lem:qlcd-net}}D}/{\sqrt n}\right)^m$.

On the complementary coordinates, we only use
$\|x|_{T^c}\|_2\le1$. The usual volumetric estimate gives
\[
\begin{aligned}
N\left(B_2^{n-m},\frac{4\mu\sqrt n}{D}B_2^{n-m}\right)\le\left(
1+\frac{D}{2\mu\sqrt n}\right)^{n-m}\le\left(\frac{C_{\ref{lem:qlcd-net}}D}{\mu\sqrt n}\right)^{n-m},
\end{aligned}
\]
where the last inequality holds since
$D\ge c_{\ref{lem:qlcd-net}}\sqrt n$ and $\mu$ is sufficiently small.

Taking the product of the two component nets and using the orthogonality of
the coordinate subspaces corresponding to $T$ and $T^c$, we obtain a
 $(C_{\ref{lem:qlcd-net}} {\mu\sqrt n}/{D})$-net for the corresponding vectors. For a fixed support $T$, its cardinality
is at most
\[
\mu^{-1}D\left(\frac{C_{\ref{lem:qlcd-net}}D}{\sqrt n}\right)^m\left(
\frac{C_{\ref{lem:qlcd-net}}D}{\mu\sqrt n}\right)^{n-m}=\mu^{-1}D\left(
\frac{C_{\ref{lem:qlcd-net}}D}{\sqrt n}\right)^n\mu^{-(n-m)} \le\mu^{-1}D\left(\frac{C_{\ref{lem:qlcd-net}}D\mu^{\lambda-1}}{\sqrt n}\right)^n,
\]
where we used $m\ge\lambda n$ and $0<\mu<1$. Taking the union over the at most $n$ possible supports $T$ gives
\[
        N\left(K_D,C_{\ref{lem:qlcd-net}}\frac{\mu\sqrt n}{D}B_2^n\right)
        \le n\mu^{-1}D
        \left(\frac{C_{\ref{lem:qlcd-net}}D\mu^{\lambda-1}}{\sqrt n}\right)^n.
\]

Apply Lemma~\ref{lem:rr} with a fixed $\alpha \in(0,1/2)$ and
$\beta =C_{\ref{lem:qlcd-net}}\mu\sqrt n/D$. Since
$D\ge c_{\ref{lem:qlcd-lower}}\sqrt n$, we have
$\beta <\alpha /10$ after choosing $\mu$ sufficiently small.
The additional factor
$\exp(C_{\ref{lem:rr}}\alpha ^{0.08}\log(1/\alpha)n)$ is absorbed
into the constant in the base of the preceding covering estimate.
The bound for $\|M(x-y)\|_2$ follows from the same lemma. Increasing $C_{\ref{lem:qlcd-net}}$ completes the proof.
\end{proof}

For $\kappa>0$, set $L:=\mu n/\kappa$. By Lemma~\ref{lem:qlcd-lower}, the QLCD of every incompressible vector is bounded below by $c_{\ref{lem:qlcd-lower}}\sqrt n$. Therefore, decomposing the remaining vectors according to their dyadic QLCD levels and writing $D=2^d$, we obtain
\[
\begin{aligned}
&\mathbf P_\Xi\left\{
\exists x\in\operatorname{Incomp}(\delta,\rho):
\|Bx\|_2\le\kappa
\right\} \\&\le
\sum_{\substack{
d=\log (c_{\ref{lem:qlcd-lower}}\sqrt n)}}^{\log(\mu n/\kappa)}
\mathbf P_\Xi\left\{
\exists x\in K_D:
\|Bx\|_2\le\kappa
\right\}\\ &\qquad+
\mathbf P_\Xi\left\{
\exists x\in\operatorname{Incomp}(\delta,\rho):
\|Bx\|_2\le\kappa,\ 
\QLCD^{\mathcal T}_{Q,\mu\sqrt n,\gamma}(x)\ge L
\right\}.
\end{aligned}
\]
Therefore, the proof is reduced to the small QLCD regime $c_{\ref{lem:qlcd-lower}}\sqrt n\le D\le L$ and the large QLCD regime
 $\QLCD^{\mathcal T}_{Q,\mu\sqrt n,\gamma}(x)\ge L$. 

\subsection{The small QLCD range}
\begin{proposition}\label{prop:small}
For $\chi\ll\mu\ll1$, there exists a constant $c_{\ref{prop:small}}>0$ such that, for every fixed $A\in\mathcal G$ and every $\kappa\ge e^{-\chi n}$,
\[
\mathbf P_\Xi\left\{
        \exists x\in\operatorname{Incomp}(\delta,\rho):
        \|Bx\|_2\le\kappa,\ 
        \QLCD^{\mathcal T}_{Q,\mu\sqrt n,\gamma}(x)\le L
\right\}
\le
e^{-c_{\ref{prop:small}}n}.
\]
\end{proposition}
\begin{proof}
By Lemma~\ref{lem:qlcd-lower}, 
$\QLCD^{\mathcal T}_{Q,\mu\sqrt n,\gamma}(x)\ge c_{\ref{lem:qlcd-lower}}\sqrt n$ for every $x\in\operatorname{Incomp}(\delta,\rho)$. If
$L<c_{\ref{lem:qlcd-lower}}\sqrt n$, the event in the statement is empty.
We may therefore decompose the remaining range into dyadic level sets
$K_D$, where $c_{\ref{lem:qlcd-lower}}\sqrt n\le D\le L$. Since $\kappa\ge e^{-\chi n}$, the number of such levels is at most $\chi n\log(\mu n)$.

Fix one such $D$ and let $\mathcal N_D$ be the net given by
Lemma~\ref{lem:qlcd-net}. If there exists $x\in K_D$ such that
$\|Bx\|_2\le\kappa$, then there is $y\in\mathcal N_D$ satisfying 
        $\|B(x-y)\|_2
        \le
        C_{\ref{lem:qlcd-net}}
         {\mu}\|B\|_{\mathrm{HS}}/D
        \le
        C_{\ref{lem:qlcd-net}} {\mu n}/{D}$. 
Thus, 
        $\|By\|_2
        \le
        \kappa+
        C_{\ref{lem:qlcd-net}} {\mu n}/{D}$. 

Since $\mathcal N_D\subset K_D+ (C_{\ref{lem:qlcd-net}} {\mu\sqrt n}/{D})B_2^n$, we obtain a uniform small ball estimate for the points of
$\mathcal N_D$. For every $y\in\mathcal N_D$, there exists $z=z(y)\in K_D$
such that $\|y-z\|_2 \le C_{\ref{lem:qlcd-net}} {\mu\sqrt n}/{D}$.
Since $\QLCD^{\mathcal T}_{Q,\mu\sqrt n,\gamma}(z)\ge D$,
at most $Q-1$ row supports satisfy $\LCD_{\mu\sqrt n,\gamma}(z|_{T_i})<D$.
Moreover, Lemma~\ref{lem:qlcd-lower} shows that fewer than $Q$ row supports fail the bound $\|z|_{T_i}\|_2\ge a_{\ref{lem:qlcd-lower}}$.
Thus, outside a set of fewer than $2Q$ restrictions,
\[
        \|z|_{T_i}\|_2\ge a_{\ref{lem:qlcd-lower}},
        \qquad
        \LCD_{\mu\sqrt n,\gamma}(z|_{T_i})\ge D.
\]
For each such restriction, $\|y|_{T_i}-z|_{T_i}\|_2\le C_{\ref{lem:qlcd-net}}{\mu\sqrt n}/{D}$. 
Since $D\ge c_{\ref{lem:qlcd-lower}}\sqrt n$, choosing $\mu$ sufficiently small gives $\|y|_{T_i}\|_2\ge a_{\ref{lem:qlcd-lower}}/{2}$.
Lemma~\ref{lem:lcd-stability} also yields $\LCD_{\mu\sqrt n/2,\gamma/2}(y|_{T_i})\gtrsim D$.

For every such row, 
\[
(By)_i=\sum_{j\in T_i}\xi_{ij}y_j.
\]
Therefore, by Lemma~\ref{lem:lcd-smallball}, for every
$\varepsilon\ge0$, there exist constants $C,c>0$ such that
\[
        \mathcal L\big((By)_i,\varepsilon\big)
        \le
        C\varepsilon+\frac{C}{D}+Ce^{-c\mu^2n}.
\]
Set
\[
        \varepsilon_D:=
        \frac{\kappa+C_{\ref{lem:qlcd-net}}\mu n/D}{\sqrt n}.
\]
Since $D\le L=\mu n/\kappa$, we have $\varepsilon_D\lesssim{\mu\sqrt n}/{D}$. 
Also, since $D\ge c_{\ref{lem:qlcd-lower}}\sqrt n$ and $D\le L\le\mu ne^{\chi n}$ with
$\chi\ll\mu$, we have 
        $ 1/D+\exp(-c\mu^2n)
        \lesssim {\mu\sqrt n}/{D}$. 
Hence, Lemma~\ref{lem:lcd-smallball} yields
\[
        \mathcal L((By)_i,\varepsilon_D)
        \lesssim
        \frac{\mu\sqrt n}{D}
\]
for all but fewer than $2Q$ rows.

Conditioned on the fixed support matrix $A$, the rows of $B$ are independent.
A straightforward conditional version of the tensorization inequality \cite[Lemma 2.2(1)]{MR2407948} shows that, for an absolute constant $C'>0$,
\[
\mathbf P_\Xi\left\{
        \|By\|_2
        \le
        \kappa+
        C_{\ref{lem:qlcd-net}}\frac{\mu n}{D}
\right\}
\le
\left(
        C'\frac{\mu\sqrt n}{D}
\right)^{n-2Q}.
\]
Taking the union bound over $\mathcal N_D$ and using
Lemma~\ref{lem:qlcd-net}, we obtain
\[
\begin{aligned}
\mathbf P_\Xi\left\{
        \exists x\in K_D:\|Bx\|_2\le\kappa
\right\}
&\le
n\mu^{-1}D
\left(
        \frac{
        C_{\ref{lem:qlcd-net}}D\mu^{\lambda-1}
        }{\sqrt n}
\right)^n
\left(
        C'\frac{\mu\sqrt n}{D}
\right)^{n-2Q}                                                   \\
&\le
n\mu^{-2Q-1}D^{2Q+1}C'^{-2Q}
\left(C'C_{\ref{lem:qlcd-net}}\mu^\lambda\right)^n .
\end{aligned}
\]

Choose $\mu>0$ sufficiently small so that $C'C_{\ref{lem:qlcd-net}}\mu^\lambda<1$. 
Since $D\le L\le \mu ne^{\chi n}$, choosing $\chi\ll\mu\ll1$ so that $e^{(2Q+1)\chi}C'C_{\ref{lem:qlcd-net}}\mu^\lambda<1$, we have
\[
        \mathbf P_\Xi\left\{
        \exists x\in K_D:\|Bx\|_2\le\kappa
        \right\}
        \le e^{-c'n},
\]
for some $c'>0$ depending only on the fixed parameters. Summing over all levels completes the proof.
\end{proof}
\subsection{The large QLCD range}
We now turn to the remaining regime
\[
        \QLCD^{\mathcal T}_{Q,\mu\sqrt n,\gamma}(x)
        \ge
        L=\frac{\mu n}{\kappa}.
\]
For every possible realization $M$ of $B$, fix vectors
$x(M),y(M)\in\mathbb S^{n-1}$ satisfying
\[
        \|Mx(M)\|_2=s_n(M),
        \qquad \|M^Ty(M)\|_2=s_n(M).
\]

Define
\[
\mathcal E_{\mathrm{large}}
:=
\left\{
        s_n(B)\le\kappa,\ 
        x(B),y(B)\in\operatorname{Incomp}(\delta,\rho),\
        \QLCD^{\mathcal T}_{Q,\mu\sqrt n,\gamma}(x(B))\ge L
\right\}.
\]
The treatment of the large QLCD regime is reduced to the following
proposition.

\begin{proposition}\label{prop:large}
There exist constants $C_{\ref{prop:large}},c_{\ref{prop:large}}>0$ such that for every
$\kappa>0$ and every fixed $A\in\mathcal G$,
\[
        \mathbf P_\Xi(\mathcal E_{\rm large})
        \le C_{\ref{prop:large}}\kappa\sqrt n+e^{-c_{\ref{prop:large}}n}.
\]
\end{proposition}
By Lemma~\ref{lem:spread}, on
$\mathcal E_{\rm large}$ at least $\rho^2\delta n/2$ indices satisfy
$|y_i(B)|\ge\rho/\sqrt{2n}$. Fewer than $Q$ indices satisfy $\|x(B)|_{T_i}\|_2<a_{\ref{lem:qlcd-lower}}$, by
Lemma~\ref{lem:qlcd-lower}, and fewer than $Q$ indices satisfy
$\LCD_{\mu\sqrt n,\gamma}(x(B)|_{T_i})<L$.
Hence at least $\rho^2\delta n/4$ indices satisfy
\begin{equation}\label{eq:large-good-rows}
        |y_i(B)|\ge\frac\rho{\sqrt{2n}},
        \qquad \|x(B)|_{T_i}\|_2\ge a_{\ref{lem:qlcd-lower}},
        \qquad \LCD_{\mu\sqrt n,\gamma}(x(B)|_{T_i})\ge L
\end{equation}
for all sufficiently large $n$. Define
\[
        G_i:=\mathcal E_{\rm large}\cap
        \left\{|y_i(B)|\ge\frac\rho{\sqrt{2n}},\
        \|x(B)|_{T_i}\|_2\ge a_{\ref{lem:qlcd-lower}},\
        \LCD_{\mu\sqrt n,\gamma}(x(B)|_{T_i})\ge L\right\}.
\]
Then, we have
\begin{equation}\label{eq:large-average}
        \mathbf1_{\mathcal E_{\rm large}}
        \le\frac4{\rho^2\delta n}\sum_{i=1}^n\mathbf1_{G_i},
        \qquad
        \mathbf P_\Xi(\mathcal E_{\rm large})
        \le\frac4{\rho^2\delta n}\sum_{i=1}^n\mathbf P_\Xi(G_i).
\end{equation}
Therefore, the argument is reduced to proving the following lemma.

\begin{lemma}\label{lem:single-row}
There exist constants $C_{\ref{lem:single-row}},c_{\ref{lem:single-row}}>0$ such that for every
$\kappa>0$, every fixed $A\in\mathcal G$, and every $i\in[n]$,
\[
        \mathbf P_\Xi(G_i)
        \le C_{\ref{lem:single-row}}\kappa\sqrt n
        +e^{-c_{\ref{lem:single-row}}n}.
\]
\end{lemma}
\begin{proof}[Proof of Proposition~\ref{prop:large}]
By \eqref{eq:large-average} and Lemma~\ref{lem:single-row},
\[
        \mathbf P_\Xi(\mathcal E_{\rm large})
        \le\frac4{\rho^2\delta n}\sum_{i=1}^n\mathbf P_\Xi(G_i)
        \le C\kappa\sqrt n+Ce^{-cn}.
\]
Thus, the result holds for all sufficiently large \(n\), provided that $C_{\ref{prop:large}}$ is sufficiently large and $c_{\ref{prop:large}}$ is sufficiently small.
\end{proof}
Before proving the lemma, we give the deterministic distance estimate
used in the conditioning argument. It is a simplification of the
argument in \cite[Lemma~4.3]{MR3936157}; see also
\cite[Lemma~4.2]{MR4523249}.

\begin{lemma}\label{lem:approx-normal}
Let $M=\begin{pmatrix}u\\N\end{pmatrix}$ be an $n\times n$ matrix, where $u\in\bbR^n$ and $N\in\bbR^{(n-1)\times n}$.
Let $V\subset\mathbb R^n$ be the subspace spanned by the rows of $N$, and set $s:=s_n(M)$. If $s>0$, for every $w\in\mathbb S^{n-1}$, we have
\[
        \dist(u,V)\ge\frac{s|\langle u,w\rangle|}{s+\|Nw\|_2}.
\]
Hence, if $\|Nw\|_2\le s$, then $\dist(u,V)\ge |\langle u,w\rangle|/2$.
\end{lemma}

\begin{proof}
Assume that $s>0$. Then $M$ is invertible and $u\notin V$.
Set  $v:= {P_{V^\perp}u}/{\|P_{V^\perp}u\|_2}$,
         $r:= {\langle u,w\rangle}/{\langle u,v\rangle}$. 
Hence $\|v\|_2=1$, $Nv=0$, and $\langle u,v\rangle=\dist(u,V)>0$. Since $\langle u,w-rv\rangle=0$,
\[
        M(w-rv)=\begin{pmatrix}0\\Nw\end{pmatrix},
        \qquad s\|w-rv\|_2\le\|Nw\|_2.
\]
It follows that
\[
        \frac{|\langle u,w\rangle|}{\dist(u,V)}
        =|r|\le\|w-rv\|_2+\|w\|_2
        \le1+\frac{\|Nw\|_2}{s}.
\]
This proves both estimates when $s>0$. If $s=0$ and $\|Nw\|_2\le s$, then $Nw=0$ and $w\in V^\perp$.
Therefore,
\[
        |\langle u,w\rangle|
        =|\langle P_{V^\perp}u,w\rangle|
        \le\|P_{V^\perp}u\|_2=\dist(u,V),
\]
which completes the proof.
\end{proof}

\begin{proof}[Proof of Lemma~\ref{lem:single-row}]
Fix $i\in[n]$, and let $\mathcal H_i$ be the $\sigma$-algebra generated by
the rows $R_k(B)$ with $k\ne i$. Conditional on $\mathcal H_i=H$, the only
remaining randomness is the $i$-th row, and $R_i|_{T_i}=(\xi_{ij})_{j\in T_i}$, 
whose entries are independent Rademacher variables.

Let $\mathcal C_H$ be the finite set of matrices obtained by varying the $i$-th row while keeping all other rows fixed according to $H$, and define
\[
        \mathcal S_H:=\{M\in\mathcal C_H:
        \|x(M)|_{T_i}\|_2\ge a_{\ref{lem:qlcd-lower}},\
        \LCD_{\mu\sqrt n,\gamma}(x(M)|_{T_i})\ge L\}.
\]
If $\mathcal S_H=\varnothing$, then $G_i$ cannot occur under this condition. Otherwise, choose $\widetilde B_H\in\mathcal S_H$ such that
 $s_n(\widetilde B_H)=\min_{M\in\mathcal S_H}s_n(M)$,
          $w_H:=x(\widetilde B_H)$. 
Then $w_H$ is $\mathcal H_i$-measurable and satisfies $\|w_H|_{T_i}\|_2\ge a_{\ref{lem:qlcd-lower}}$,
         $\LCD_{\mu\sqrt n,\gamma}(w_H|_{T_i})\ge L$. 
         
Consider a matrix $M\in\mathcal C_H$ for which the event $G_i$ occurs. Then $M\in\mathcal S_H$, and hence $s_n(\widetilde B_H)\le s_n(M)$.
Let $N_H$ be the matrix obtained by deleting the $i$-th row from $\widetilde B_H$. Then, we have
\[
        \|N_Hw_H\|_2\le\|\widetilde B_Hw_H\|_2
        =s_n(\widetilde B_H)\le s_n(M).
\]
Let $V_H\subset\mathbb R^n$ be the linear subspace spanned by the rows of $N_H$. After permuting the rows so that $R_i(B)$ becomes the first row, Lemma~\ref{lem:approx-normal} gives $\dist(M_i,V_H)\ge |\langle M_i,w_H\rangle|/2$. 

On the other hand, since $B\in G_i$, we have
$s_n(M)\le\kappa$, $|y_i(M)|\ge\rho/\sqrt{2n}$. Moreover, since
$\sum_{k\ne i}y_k(M)M_k\in V_H$, we obtain
\[
        |y_i(M)|\dist(M_i,V_H)
        =\dist(y(M)^TM,V_H)
        \le\|M^Ty(M)\|_2=s_n(M)\le\kappa.
\]
Therefore,
\[
        |\langle M_i,w_H\rangle|\le2\dist(M_i,V_H)\le\frac{2\sqrt2}{\rho}\kappa\sqrt n.
\]
Thus, by Lemma~\ref{lem:lcd-smallball},
\[
\begin{aligned}
\mathbf P_\Xi(G_i\mid\mathcal H_i=H)
\le\mathbf P_\Xi\left\{
        \left|\sum_{j\in T_i}\xi_{ij}(w_H)_j\right|
        \le\frac{2\sqrt2}{\rho}\kappa\sqrt n
        \,\middle|\mathcal H_i=H\right\} \le C\kappa\sqrt n+\frac CL+Ce^{-c\mu^2n}.
\end{aligned}
\]
Since $1/L = {\kappa}/{\mu n} \le C\kappa\sqrt n$ for all sufficiently large $n$, there exist constants $C_{\ref{lem:single-row}}, c_{\ref{lem:single-row}}>0$ such that
\[
\mathbf P_\Xi(G_i\mid\mathcal H_i=H)
\le
C_{\ref{lem:single-row}}\kappa\sqrt n + e^{-c_{\ref{lem:single-row}}n}.
\]
The bound holds for all $H$, and integration over $\mathcal H_i$ completes the proof.
\end{proof}

\subsection{Completion of the proof}

We are now ready to complete the proof of Theorem~\ref{thm:ssv}.

\begin{proof}[Proof of Theorem~\ref{thm:ssv}]
For every constant $c_{\ref{lem:compressible}}>0$, the desired probability bound is trivial when $\kappa>c_{\ref{lem:compressible}}\sqrt n$, for all sufficiently large $n$. Therefore, it remains to consider $\exp{(-\chi n)}
        \le
        \kappa
        \le
        c_{\ref{lem:compressible}}\sqrt n$. 
Fix $A\in\mathcal G$ and set $B=A\circ\Xi$. Lemma~\ref{lem:compressible} gives
\[
\mathbf P_\Xi\left\{
        \inf_{x\in\operatorname{Comp}(\delta,\rho)}
        \|Bx\|_2
        \le
        c_{\ref{lem:compressible}}\sqrt n
\right\} +
\mathbf P_\Xi\left\{
        \inf_{y\in\operatorname{Comp}(\delta,\rho)}
        \|B^Ty\|_2
        \le
        c_{\ref{lem:compressible}}\sqrt n
\right\}
\le
2e^{-c_{\ref{lem:compressible}}n}.
\]
Outside this event, the selected vectors $x(B)$ and $y(B)$ both belong to $\operatorname{Incomp}(\delta,\rho)$.

Lemma~\ref{lem:qlcd-lower} gives $\QLCD^{\mathcal T}_{Q,\mu\sqrt n,\gamma}(x(B)) \ge c_{\ref{lem:qlcd-lower}}\sqrt n$. 
It follows that either \[
c_{\ref{lem:qlcd-lower}}\sqrt n
        \le
        \QLCD^{\mathcal T}_{Q,\mu\sqrt n,\gamma}(x(B))
        \le L, 
\, \text{or}\, \QLCD^{\mathcal T}_{Q,\mu\sqrt n,\gamma}(x(B))
        \ge L. \] 
The first case is bounded by Proposition~\ref{prop:small}, while the second case is contained in $\mathcal E_{\rm large}$ and is covered by Proposition~\ref{prop:large}. Thus, there exist constants $C,c>0$ such that
\[
\mathbf P_\Xi\{s_n(B)\le\kappa\}
\le
2e^{-c_{\ref{lem:compressible}}n}
+
e^{-c_{\ref{prop:small}}n}+
C_{\ref{prop:large}}\kappa\sqrt n
+
e^{-c_{\ref{prop:large}}n}\le
C\kappa\sqrt n
+
e^{-cn},
\]
for sufficiently large $n$. This bound holds for all deterministic $A\in\mathcal G$. Hence, using Proposition~\ref{prop:support}, there exist constants $C_{\lambda},c_{\lambda}>0$ such that
\[
\begin{aligned}
\mathbf P_{A,\Xi}\{s_n(A\circ\Xi)\le\kappa\}
 \le
\mathbf P_A(\mathcal G^c)
+
\sup_{A\in\mathcal G}
\mathbf P_\Xi\{s_n(A\circ\Xi)\le\kappa\}  \le
C_{\lambda}\kappa\sqrt n
+
e^{-c_{\lambda}n}.
\end{aligned}
\]
This completes the proof.
\end{proof}

\begingroup
\section*{Statement on the Use of Artificial Intelligence}
All principal ideas, results, and proof strategies in this paper were
developed by the authors. In particular, the proof of the smallest
singular value estimate and the conceptual framework of the rank
argument---including the covering-defect method, the sparse--nonsparse
kernel decomposition, and the thin--thick argument---were developed by
the authors without AI assistance. During the preparation of the
manuscript, ChatGPT and Codex (OpenAI), including the GPT-5.6 Sol and
GPT-6 Pro models, were used only as auxiliary tools to help check and
refine some technical details after the underlying ideas and proof structure had been supplied by the authors, and to assist with language editing. All AI-assisted material included in the manuscript was independently reviewed and verified by the authors, who take full responsibility for the mathematical content and final text.
\endgroup

\begingroup
\section*{Acknowledgments}
This work began with a discussion between K. Yu and Hanchao Wang. H. Liang, Y. Shi, and K. Yu thank Hanchao Wang for his guidance and valuable suggestions. T. Yu is grateful to his advisor, Alexander Litvak, for his continued support and encouragement, and for many helpful suggestions.
\endgroup

\bibliographystyle{amsplain}
\bibliography{references_1}

\vspace{1.2cm}

\par\bigskip
\begingroup
\small
\setlength{\parindent}{0pt}

\textsc{Huaijin Liang}\par\nobreak
\smallskip
Institute for Financial Studies,
Shandong University, Jinan 250100, China.\par\nobreak
\smallskip
\textit{Email address:} \texttt{lianghuaij@gmail.com}\par

\bigskip

\textsc{Yuxiao Shi}\par\nobreak
\smallskip
Institute for Financial Studies,
Shandong University, Jinan 250100, China.\par\nobreak
\smallskip
\textit{Email address:} \texttt{yshi6821@gmail.com}\par

\bigskip

\textsc{Kexin Yu}\par\nobreak
\smallskip
School of Mathematics, Shandong University,
Jinan 250100, China.\par\nobreak
\smallskip
\textit{Email address:} \texttt{kexinyu01@gmail.com}\par

\bigskip

\textsc{Tingzhou Yu}\par\nobreak
\smallskip
Department of Mathematical and Statistical Sciences,
University of Alberta, Edmonton, AB T6G 2R3, Canada.\par\nobreak
\smallskip
\textit{Email address:} \texttt{tingzho1@ualberta.ca}\par

\endgroup

\end{document}